\documentclass{amsart}
\usepackage[fontsize=12pt]{scrextend}
\usepackage{nicefrac}
\usepackage{enumitem}
\usepackage{graphicx}
\usepackage{mathrsfs}
\usepackage{xcolor}
\usepackage{bbm}
\usepackage[hidelinks]{hyperref}
\usepackage{cite}
\usepackage[normalem]{ulem}

\newtheorem{theorem}{Theorem}[section]
\newtheorem{proposition}[theorem]{Proposition}
\newtheorem{corollary}[theorem]{Corollary}
\newtheorem{lemma}[theorem]{Lemma}
\newtheorem{sublemma}[theorem]{Sublemma}
\newtheorem{subsublemma}[theorem]{Subsublemma}

\newtheorem*{thm}{Theorem}

\theoremstyle{definition}
\newtheorem{definition}[theorem]{Definition}

\theoremstyle{remark}
\newtheorem{remark}[theorem]{Remark}

\numberwithin{equation}{section}

\allowdisplaybreaks[4]

\begin{document}
	
	\sloppy
\title[Frequently supercyclic operators and an $\mathscr{F}_\Gamma$-hypercyclicity criterion]{On frequently supercyclic operators and an $\mathscr{F}_\Gamma$-hypercyclicity criterion with applications}
%\title{Frequent, $\mathscr{U}$-frequent, and Reiterate $\Gamma$-Supercyclicity??? (alterar!)}
\author{Thiago R. Alves}
\address{Departamento de Matem\'{a}tica,
	Instituto de Ci\^{e}ncias Exatas,
	Universidade Federal do Amazonas,
	69.080-900 -- Manaus -- Brazil}
\email{alves@ufam.edu.br}
\thanks{The first author was financed in part by the Coordena\c{c}\~{a}o de Aperfei\c{c}oamento de Pessoal de N\'{\i}vel Superior – Brasil (CAPES) – Finance Code 001 and FAPEAM}

\author{Geraldo Botelho}
\address{Instituto de Matem\'atica e Estat\'istica,
	Universidade Federal de Uberlândia,
	38.400-902 -- Uberlândia -- Brazil}
\email{botelho@ufu.br}
\thanks{}

\author{Vinícius V. Fávaro}
\address{Instituto de Matem\'atica e Estat\'istica,
	Universidade Federal de Uberlândia,
	38.400-902 -- Uberlândia -- Brazil}
\email{vvfavaro@ufu.br}
\thanks{The second and third authors were partially supported by FAPEMIG Grant APQ-01853-23}

%    General info

\date{}

\dedicatory{}


\begin{abstract}
Given a Furstenberg family $\mathscr{F}$ and a subset $\Gamma \subset \mathbb{C}$, we introduce and explore the notions of $\mathscr{F}_{\Gamma}$-hypercyclic operator and $\mathscr{F}$-hypercyclic scalar set. First, the study of $\mathscr{F}_\mathbb{C}$-hypercyclic operators yields new interesting information about frequently supercyclic, $\mathscr{U}$-frequently supercyclic, reiteratively supercyclic and supercyclic operators. Then we provide a criterion for identifying $\mathscr{F}_{\Gamma}$-hypercyclic operators. As applications of this criterion, we show that any unilateral pseudo-shift operator on $c_0(\mathbb{N})$ or $\ell_p(\mathbb{N})$, $1 \leq p < \infty$, is $\mathscr{F}_{\Gamma}$-hypercyclic for every unbounded subset $\Gamma$ of $\mathbb{C}$. Moreover, under the same condition on $\Gamma$, we show that any separable infinite-dimensional Banach space supports an $\mathscr{F}_{\Gamma}$-hypercyclic operator. Finally, our study provides sufficient and necessary conditions for a subset $\Gamma \subset \mathbb{C}$ to be a hypercyclic scalar set. These results give partial answers to a question raised by Charpentier, Ernst, and Menet in 2016.

%We introduce and explore the notions of frequent, $\mathscr{U}$-frequent, and reiterative $\Gamma$-supercyclicity. We provide a criterion for identifying operators that meet  \textcolor{red}{(have? satisfy?)} these notions. As applications of such criteria, we show that every weighted shift operator on $c_0(\mathbb{N})$ and $\ell_p(\mathbb{N})$, $1 \leq p < \infty$, is frequently supercyclic, and any separable infinite-dimensional Banach space supports a frequently supercyclic operator. Our study yields a necessary condition for frequent, $\mathscr{U}$-frequent, and reiterative hypercyclic scalar sets and provides a sufficient condition for $\mathscr{U}$-frequent and reiterative hypercyclic scalar sets. These results collectively offer partial solutions to a question raised by Charpentier, Ernst, and Menet.

\end{abstract}

\maketitle

\section{Introduction and overview of the main results}
\label{Sc-1}

Throughout this article,  $\mathbb{N}$ denotes the set of positive integers, $\mathbb{N}_0 = \mathbb{N} \cup \{0\}$ and $\mathscr{P}(\mathbb{N}_0)$ denotes the family of all subsets of $\mathbb{N}_0$. By \textit{operator} we mean a \textit{bounded linear operator}, and all Banach spaces are supposed to be complex, separable and  infinite-dimensional. As usual, the symbol $\mathscr{L}(X)$ represents the space of all operators acting on a Banach space $X$.

Recalling a classic notion in the field, we say that an operator $T$ on a Banach space $X$ is \textit{hypercyclic} if there exists $x \in X$ whose orbit under $T$, given by
$$\mbox{Orb}(x,T) := \{x, Tx, T^2x, \ldots\},$$
is dense in $X$. In this case, $x$ is said to be a \textit{hypercyclic vector} for $T$. Hypercyclicity is a focal point in the study of linear dynamics. Beyond hypercyclicity, concepts like \textit{frequent hypercyclicity}  \cite{BayGri2004, BayGri2006} and \textit{$\mathscr{U}$-frequent hypercyclicity} \cite{Shkarin09} have been introduced. Roughly speaking, these notions impose stronger conditions by requiring that the orbit of the operator visits each open set of the space with a frequency denser than merely infinitely many times. Precise definitions shall be given soon. %We will delve deeper into these distinctions in the upcoming discussion \textcolor{red}{(confirmar o inglês desta frase)}.

The \textit{lower} and \textit{upper asymptotic density} of a subset $E$ of $\mathbb N_0$  are defined as
\begin{align}
	\begin{split} \label{upper-dens-natural}
		\underline{d}(E) := \liminf_{n \to \mathbb \infty} \dfrac{\#(E \cap [0,n])}{n} \ \ \mbox{ and } \ \ \overline{d}(E) := \limsup_{n \to \mathbb \infty} \dfrac{\#(E \cap [0,n])}{n},
	\end{split}
\end{align}
respectively. Given an open subset $U$ of $X$, an operator $T$ on $X$, and $x\in X$, we denote
\begin{align*}
	\mathcal{N}_T(x,U) := \{n \in \mathbb N_0 : T^nx \in U\},
\end{align*}
where $T^0x := x$. If each set $\mathcal{N}_T(x, U)$ is non-empty for every open subset $U$ of $X$, then each such set must actually be infinite. In this case, as discussed previously, $T$ is hypercyclic. Now, if the orbit of $x$ under $T$ visits each open set $U$ more frequently, in the sense that
$$\mbox{$\underline{d}(\mathcal{N}_T(x,U)) > 0$ \ \ (resp. $\overline{d}(\mathcal{N}_T(x,U)) > 0$)},$$
then $T$ is called \textit{frequently hypercyclic} (resp. \textit{$\mathscr{U}$-frequently hypercyclic}).

The operator $T$ is called \textit{supercyclic} if there exists $x$ in $X$ such that the set
\begin{eqnarray} \label{aux-equation-39}
	\mathbb{C} \cdot \mbox{Orb}(x,T) := \{\lambda \cdot T^nx : \lambda \in \mathbb C, \, n \in \mathbb N_0\}
\end{eqnarray}
is dense in $X$. It is evident that all hypercyclic operators are also supercyclic. However, it is well known that there are supercyclic operators that are not hypercyclic (cf. \cite[Example 1.15]{BayMath-book}).

Significant changes occur when the set in (\ref{aux-equation-39}) is constrained by considering a proper subset $\Gamma$ of $\mathbb{C}$ rather than $\mathbb{C}$ itself. For instance, given $x \in X$, $T \in \mathscr{L}(X)$ and a non-empty finite set $\Gamma_0 \subset \mathbb{C} \setminus \{0\}$, a result independently established by Peris and Costakis (cf. \cite[Theorem 4]{Peris01} and \cite[Theorem 1]{Costakis02})  asserts that the set  $\Gamma_0 \cdot Orb(x, T)$ is dense in $X$ if and only if $Orb(x, T)$ is dense in $X$. Similarly, in \cite[Corollary 2]{Leo-SaaMul04} it is proved that for the unit circle $\mathbb{T} := \{\lambda \in \mathbb{C}: |\lambda| = 1\}$, the denseness of $\mathbb{T} \cdot \mbox{Orb}(x, T)$ in $X$ is equivalent to the denseness of the orbit of $x$ under $T$. This kind of results naturally prompts the following question:
	\begin{quote}
	What are the precise conditions that must be imposed on $\Gamma$ to ensure that the denseness of $\Gamma \cdot \text{Orb}(x, T)$ in $X$ implies that  $x$ is a hypercyclic vector for $T$?
	\end{quote}
This question was completely answered in \cite{CharErnMenet16}. Buil\-ding upon \cite{CharErnMenet16}, we define $T$ as \textit{$\Gamma$-supercyclic} if there exists $x$ in $X$ such that the set $\Gamma \cdot \mbox{Orb}(T, x)$ is dense in $X.$ And a non-empty subset $\Gamma$ of $\mathbb{C}$ is called a \textit{hypercyclic scalar set} if, regardless of the  Banach space $Y,$ the operator $S$ on $Y$ and the vector $y$ in $Y,$ it holds that $\overline{\Gamma \cdot \mbox{Orb}(y, S)} = Y$ if and only if $y$ is a hypercyclic vector for $S$. The following result was proved in \cite{CharErnMenet16}:%, it was established the following theorem:

\begin{thm}{\rm \cite[Theorem A]{CharErnMenet16}}
	A non-empty subset $\Gamma$ of $\mathbb{C}$ is a hypercyclic scalar set if and only if $\Gamma \setminus \{0\}$ is non-empty, bounded, and bounded away from zero.
\end{thm}

Wondering about related results, the authors also questioned what could be proved within the realm of frequent and $\mathscr{U}$-frequent hypercyclicity (cf. \cite[Question 9]{CharErnMenet16} and also \cite[Question 10]{CharErn2020}). With the purpose of addressing this question, we introduce and explore the concepts outlined in the following definition.

\begin{definition} \label{def-hypset-and-Fhypope}
	\rm	Let $\Gamma \subset \mathbb C$ and let $\mathscr{F} \subset \mathscr{P}(\mathbb{N}_0)$ be a Furstenberg family, meaning that $\mathscr{F}$ is non-trivial, that is, $ \emptyset \notin \mathscr{F}\neq \emptyset$, and is hereditarily upward, that is, if $E \subset F$ and $E \in \mathscr{F}$, then $F \in \mathscr{F}.$
	\begin{enumerate}[label=(\Alph*)]
		\item An operator $T \in \mathscr{L}(X)$ is an \textit{$\mathscr{F}_{\Gamma}$-hypercyclic operator} if there exists $x$ in $X$ such that
		$$\mathcal{N}_T(\Gamma x, U) := \{n \in \mathbb N_0 : T^n(\Gamma x) \cap U \not= \emptyset\} \in \mathscr{F}$$
		for any non-empty open set $U$ in $X.$ In this case, $x$ is called an \textit{$\mathscr{F}_{\Gamma}$-hypercyclic vector} for $T.$
		\item The set $\Gamma$ is termed an \textit{$\mathscr{F}$-hypercyclic scalar set} if, regardless of the Banach space $Y,$ the operator $S$ on $Y,$ and the vector $y \in Y,$ it holds that $y$ is an $\mathscr{F}_\Gamma$-hypercyclic vector for $S$ if and only if $y$ is a $\mathscr{F}_{\{1\}}$-hypercyclic vector for $S.$
	\end{enumerate}
\end{definition}

Note that the set of $\mathscr{F}_{\{1\}}$-hypercyclic vectors for $T$ coincides with the set of $\mathscr{F}_{\{\delta\}}$-hypercyclic vectors for $T$ for any $\delta \neq 0$. For this reason, for any $\delta \neq 0$, a $\mathscr{F}_{\{\delta\}}$-hypercyclic operator (resp., $\mathscr{F}_{\{\delta\}}$-hypercyclic vector) will be simply referred to as an \textit{$\mathscr{F}$-hypercyclic operator} (resp., \textit{$\mathscr{F}$-hypercyclic vector}). It is worth noting that %several studies have been conducted to investigate
 this class of $\mathscr{F}$-hypercyclic operators have already been investigated for somewhat general Furstenberg families $\mathscr{F}$, see, e.g.,  \cite{BesMenetPerisPuig16,BesMenetPerisPuig19,BonillaErdmann18,BonillaErdmannMartinez22,CardMuro22-1,CardMuro22-2, GrivauxLopez-Martinez23}.

 Less general definitions of a Furstenberg family have appeared in several articles. In this work, we follow the definition provided in \cite{BesMenetPerisPuig16} (see also \cite{BayGriMathMen-preprint}). More stringent definitions can be found, for example, in \cite{BonillaErdmannMartinez22} and \cite{GrivauxLopez-Martinez23}, where it is also required that each element of the Furstenberg family $\mathscr{F}$ is infinite. In \cite{GrivauxLopez-Martinez-Peris-preprint}, it is additionally required that $A \cap [n, \infty[ \in \mathscr{F}$ whenever $n \in \mathbb{N}$ and $A \in \mathscr{F}$. Despite of being very general, with the definition we are adopting it is true that any $\mathscr{F}_\Gamma$-hypercyclic vector for an operator $T$ is $\Gamma$-supercyclic for $T$ (cf. Lemma \ref{newlemma}).

%{\color{violet} We observe that for any set $\Gamma \subset \mathbb C$, any Furstenberg family $\mathscr{F}$, and any operator $T \in \mathscr{L}(X)$, it holds:
%\begin{align} \label{F-hypimpliesG-sup}
%	\mbox{$x_0$ is $\mathscr{F}_\Gamma$-hypercyclic for $T$  $\Longrightarrow$ $x_0$ is $\Gamma$-supercyclic for $T$.}
%\end{align}
%Assume, for the sake of contradiction, that $x_0$ is not a $\Gamma$-supercyclic vector for $T$. Thus, there exists a non-empty open set $U$ such that $\mathcal{N}_T(\Gamma x_0, U)$ is finite. Define $V := U \setminus \bigcup\limits_{m \in \mathcal{N}_T(\Gamma x_0, U)} T^m(\mathbb{C} x_0)$, which remains a non-empty open set, since only finitely many straight lines are removed from $U$. It is clear that $\mathcal{N}_T(\Gamma x_0, V) = \emptyset$. However, since $x_0$ is an $\mathscr{F}_\Gamma$-hypercyclic vector for $T$, it follows that $\mathcal{N}_T(\Gamma x_0, V) \in \mathscr{F}$, leading to a contradiction because $\emptyset \not\in \mathscr{F}$.}

Important Furstenberg families in $\mathscr{P}(\mathbb{N}_0)$ consist of those formed by subsets with properties such as infinite cardinality, positive lower asymptotic density, positive upper asymptotic density, and positive upper Banach density. These families are denoted by $\mathscr{A}_\infty$, $\overline{\mathscr{D}}$, $\underline{\mathscr{D}}$, and $\overline{\mathscr{B}}$, respectively. For the sake of the reader, we recall that the upper Banach density of a subset $E$ of $\mathbb{N}_0$ is defined as
\begin{align}
	\begin{split} \label{upper-Banach-density}
		\overline{b}(E) := \lim_{s \to \infty} \left[ \limsup_{k \to \infty} \dfrac{\#(E \cap [k+1,k+s])}{s}\right].
	\end{split}
\end{align}
In the literature, $\overline{\mathscr{B}}$-hypercyclic operators are also called \textit{reiteratively hypercyclic} (cf. \cite[p. 547]{BesMenetPerisPuig16}).

The class of $\overline{\mathscr{D}}_\Gamma$-hypercyclic (resp., $\underline{\mathscr{D}}_{\Gamma}$ and ${\overline{\mathscr{B}}}_{\Gamma}$-hypercyclic) operators will be referred to as \textit{frequently $\Gamma$-supercyclic} (resp., \textit{$\mathscr{U}$-frequently $\Gamma$-supercyclic} and \textit{reiteratively $\Gamma$-supercyclic}) operators. Furthermore, if $\Gamma = \mathbb{C}$, we will simply call such operators \textit{frequently supercyclic} (resp., \textit{$\mathscr{U}$-frequently supercyclic} and \textit{reiteratively supercyclic}).

As a direct consequence of the definitions, for every $E \subset \mathbb N_0$, we have $$\underline{d}(E) \leq \overline{d}(E) \leq \overline{b}(E).$$
Therefore, the following inclusions between sets of operators on a Banach space $X$ hold:
\begin{align} \label{inclusions}
	%\small
	\begin{split}
		\left\{\substack{\mbox{frequently} \\ \mbox{supercyclic} \\ \mbox{operators}}\right\} \subset \left\{\substack{\mathscr{U}\mbox{-frequently} \\ \mbox{supercyclic} \\ \mbox{operators}}\right\} \subset \left\{\substack{\mbox{reiteratively} \\ \mbox{supercyclic} \\ \mbox{operators}}\right\} \subset \left\{\substack{\mbox{supercyclic} \\ \mbox{operators}}\right\}.
	\end{split}
\end{align}

If we write `hypercyclic' instead of `supercyclic' in \eqref{inclusions}, it is well-known that all such inclusions can be strict (cf. \cite[Theorem 5]{BayRuz2015} and \cite[Theorems 7 and 13]{BesMenetPerisPuig16}). Therefore, it is a natural question to ask whether the same happens for the inclusions in \eqref{inclusions}.

Our initial result, proved in Section \ref{Sc-2}, confirms that the inclusions in \eqref{inclusions} can indeed be strict. Also in Section \ref{Sc-2}, we establish the existence of an invertible frequently (resp. $\mathscr{U}$-frequently) supercyclic operator on a certain Banach space, whose inverse fails to be frequently (resp. $\mathscr{U}$-frequently) supercyclic. The proofs of these results are based on a theorem (Theorem \ref{sup-not-freqsup}), which is stated in the general context of $\mathscr{F}_{\Gamma}$-hypercyclic operators. This theorem essentially relates the existence (or non-existence) of an $\mathscr{F}$-hypercyclic operator on $X$ to the existence (or non-existence) of an $\mathscr{F}_{\mathbb C}$-hypercyclic operator on $X \oplus \mathbb C$.

A standard method to determine if an operator $T$ is hypercyclic (resp. frequently hypercyclic) is to apply the Hypercyclicity Criterion \cite{BesPeris99} (resp. Frequent Hypercyclicity Criterion \cite{BonillaErdmann07, BonillaErdmann09}). In \cite[Theorem 3]{BesMenetPerisPuig16}, the authors generalize these criteria to the setting of $\mathscr{F}$-hypercyclic operators. Building upon this work, in Section \ref{Sc-3} we present an $\mathscr{F}_{\Gamma}$-Hypercyclicity Criterion, which coincides with the $\mathscr{F}$-Hypercyclicity Criterion given in \cite{BesMenetPerisPuig16} when $\Gamma = \{1\}$. Moreover, we derive a corollary of this criterion %$\mathscr{F}_{\Gamma}$-Hypercyclicity Criterion,
which will be simpler to apply in the subsequent sections than the criterion itself.

In Section \ref{Sc-4} we provide two applications of the $\mathscr{F}_{\Gamma}$-Hypercyclicity Criterion. It is well-known that any unilateral weighted backward shift operator on $c_0$ or $\ell_p$, $1 \leq p < \infty$, is supercyclic. What happens %% a natural question arises as to whether the same holds
in the context of $\mathscr{F}_{\Gamma}$-hypercyclic operators? We address this question in Subsection \ref{Sub-sec-4.1} within the broader context of unilateral pseudo-shift operators. More precisely, as an application of the $\mathscr{F}_{\Gamma}$-Hypercyclicity Criterion, we show that all unilateral pseudo-shift operators are $\mathscr{F}_{\Gamma}$-hypercyclic for any unbounded subset $\Gamma$ of $\mathbb{C}$.

As a second application of the $\mathscr{F}_{\Gamma}$-Hypercyclicity Criterion, we show in Subsection \ref{Sub-sec-4.2} that, for every unbounded subset $\Gamma$ of $\mathbb{C}$,  any Banach space supports an $\mathscr{F}_{\Gamma}$-hypercyclic operator. To place this result in the landscape of linear dynamics, recall that: (i) Every Banach space supports a hypercyclic operator \cite{Ansari97, Bernal99}. (ii) There are Banach spaces that do not support $\mathscr{U}$-frequently (and thus, frequently) hypercyclic operators \cite{Shkarin09}.

Finally, in Section \ref{Sc-5}, we provide some sufficient and some necessary conditions to decide when a subset $\Gamma \subset \mathbb{C}$ is an $\mathscr{F}$-hypercyclic scalar set. Under certain assumptions on $\mathscr{F}$, we prove that $\Gamma \setminus \{0\}$ being non-empty, bounded, and bounded away from zero is sufficient for $\Gamma$ to be an $\mathscr{F}$-hypercyclic scalar set.  Moreover, we show that a necessary condition is that $\Gamma \setminus \{0\}$ is non-empty, bounded, and does not contain a sequence converging slower than any exponential function to zero. As a consequence of these general findings, we obtain some partial answers to the aforementioned question raised by Charpentier, Ernst, and Menet (cf. Remark \ref{rmk-question}).

A word about notation: $\mathbb{C}^* := \mathbb{C} \setminus \{0\}$ and the symbol $\alpha_k \searrow 0$ means that $(\alpha_k)_k$ is a decreasing sequence of real numbers converging to zero.

For deeper perspectives on the dynamics of linear operators, we refer the reader to \cite{BayMath-book, Gro-ErdMang2011-book, GrivauxMatheronMenet21}.

\section{The inclusions in (\ref{inclusions}) are all strict }\label{Sc-2}

%\section{\textcolor{blue}{Distinguishing ???} $\mathscr{F}_\mathbb{C}$-Hypercyclic Operators} \label{Sc-2}

For normed spaces $X$ and $Y$, we define $X \oplus Y$ as the normed space $$X \oplus Y := (X \times Y, {\| \cdot \|}_{X \oplus Y}), \text{ where } \, {\|(x, y)\|}_{X \oplus Y} := {\|x\|}_X + {\|y\|}_Y.$$
It is clear that $X \oplus Y$ is complete if and only if $X$ and $Y$ are. Moreover, for operators $T$ and $S$ acting on $X$ and $Y$ respectively, an operator $T \oplus S$ that acts on $X \oplus Y$ is defined by $$T \oplus S (x, y) := (Tx,Sy)  .$$
%\textcolor{blue}{\sout{Furthermore, it is evident that $T \oplus S$ is continuous if and only if $T$ and $S$ are as well.}}

First we prove a result in the general setting of $\mathscr{F}_{\Gamma}$-hypercy\-cli\-city. By $Id_{\mathbb{C}}$ we denote the identity operator on $\mathbb C$.

\begin{theorem} \label{sup-not-freqsup}
	Let $X$ be a Banach space and let $\mathscr{F} \subset \mathscr{P}(\mathbb N_0)$ be a Furstenberg family.  Then an operator $T \in \mathscr{L}(X)$ is $\mathscr{F}$-hypercyclic if and only if $T \oplus Id_{\mathbb{C}} \in \mathscr{L}(X \oplus \mathbb C)$ is ${\mathscr{F}}_{\mathbb{C}}$-hypercyclic.
\end{theorem}
\begin{proof} Assume first that
%First, let us prove the sufficiency part of the statement. Since
$T$ is $\mathscr{F}$-hypercyclic, and let $z_0 \in X$ be an  $\mathscr{F}$-hypercyclic vector  for $T$. Let us prove that $(z_0,1) \in X \oplus \mathbb C$ is an $\mathscr{F}_{\mathbb C}$-hypercyclic vector for $T \oplus Id_{\mathbb C}$. Fix $\varepsilon > 0$ and $(z,\beta) \in X \oplus \mathbb C$ with $\beta \not= 0$. Since $z_0$ is an $\mathscr{F}$-hypercyclic vector for $T$, there exists a set $E \in \mathscr{F}$ such that
$${\|T^{m} z_0 - \beta^{-1} \, z\|}_X < \beta^{-1} \, \varepsilon \ \ \mbox{for each  $m \in E$},$$
which implies
$${\|\beta \cdot (T \oplus Id_{\mathbb C})^m (z_0, 1) - (z, \beta) \|}_{X \oplus \mathbb C} < \varepsilon \ \ \mbox{for each $m \in E$}.$$
This proves that $T \oplus Id_{\mathbb{C}} $ is ${\mathscr{F}}_{\mathbb{C}}$-hypercyclic.
	
Conversely, suppose that $T \oplus Id_{\mathbb{C}} $ is ${\mathscr{F}}_{\mathbb{C}}$-hypercyclic. %Now, let us prove the necessity part of the theorem.
First, consider a countable dense subset $\mathcal{D}$ of $X$ and let ${(y_j)}_{j \geq 1}$ be a sequence in $\mathcal{D}$ such that each element of $\mathcal{D}$ appears infinitely many times in the sequence ${(y_j)}_{j \geq 1}$. The $\mathscr{F}_{\mathbb C}$-hypercyclicity of $T \oplus Id_{\mathbb C}$ implies the existence of an $\mathscr{F}_{\mathbb C}$-hypercyclic vector $(x_0,\lambda_0)$ for $T \oplus Id_{\mathbb C}$. 	Since it is evident that if $\lambda_0 = 0$, then the vector $(x_0,\lambda_0)$ is not $\mathscr{F}_{\mathbb C}$-hypercyclic for $T \oplus Id_{\mathbb C}$, we may suppose, without loss of generality, that $\lambda_0 \neq 0$. Our aim is to show that $x_0$ is an $\mathscr{F}$-hypercyclic vector for $T$. Since $(x_0,\lambda_0)$ is an $\mathscr{F}_{\mathbb{C}}$-hypercyclic vector for $T \oplus Id_{\mathbb C}$, for each positive integer $j$ there exist a subset $$E_j := \{n_1^{j} < n_2^{j} < n_3^{j} < \cdots\} \in \mathscr{F},$$ and a sequence ${(\lambda_{n_k^{j}})}_{k \geq 1}$ in $\mathbb{C}^*$ such that the following holds for every $k \geq 1$:
	\begin{align}
		\begin{split}\label{aux-exam}
			{\|\lambda_{n_k^{j}} \, T^{n_k^{j}}x_0 - y_j\|}_X + |\lambda_{n_k^{j}}\lambda_0 - \lambda_0| \leq \dfrac{1}{j}.
		\end{split}
	\end{align}

Choose $x \in \mathcal{D}$ and $\varepsilon > 0$ arbitrarily. By construction, $x$ appears infinitely many times in the sequence ${(y_j)}_{j \geq 1}$. Thus, by \eqref{aux-exam}, we can find an increasing sequence of positive integers ${(j_m)}_{m \geq 1}$ such that
	\begin{align} \label{aux-exam-1}
		{\|\lambda_{n_k^{j_m}} \, T^{n_k^{j_m}}x_0 - x\|}_X + |\lambda_{n_k^{j_m}}\lambda_0 - \lambda_0| \leq \dfrac{1}{j_m} \mbox{~for all~}  k, m \geq 1.
	\end{align}
Select a positive integer $p$ satisfying
	\begin{align}
		\begin{split}\label{aux-exam-5}
			0 < \dfrac{{\|x\|}_X + |\lambda_0|}{j_p \ |\lambda_0| - 1} < \varepsilon.
		\end{split}
	\end{align}
	Hence, $j_p \, |\lambda_0| > 1$ and, employing \eqref{aux-exam-1}, we can deduce that
	\begin{align}
		\begin{split}\label{aux-exam-4}
			|\lambda_{n_k^{j_p}}|^{-1} \leq \dfrac{j_p \, |\lambda_0|}{j_p \,|\lambda_0| - 1} \mbox{~for every~} k \geq 1.
		\end{split}
	\end{align}
Note that, for $k \geq 1$,
	\begin{align*}
		\begin{split}
			{\|T^{n_k^{j_p}}x_0 - x\|}_X &\leq |\lambda_{n_k^{j_p}}|^{-1} {\|\lambda_{n_k^{j_p}} T^{n_k^{j_p}}x_0 - x\|}_X + |\lambda_{n_k^{j_p}}|^{-1} |1 - \lambda_{n_k^{j_p}}| {\|x\|}_X \\ &\stackrel{\eqref{aux-exam-1}+\eqref{aux-exam-4}}{\leq}
			\dfrac{j_p \, |\lambda_0|}{j_p \, |\lambda_0| - 1} \cdot \dfrac{1}{j_p} + \dfrac{j_p \, |\lambda_0|}{j_p \, |\lambda_0| - 1} \cdot \dfrac{1}{j_p \, |\lambda_0|} {\|x\|}_X = \dfrac{{\|x\|}_X + |\lambda_0|}{j_p \, |\lambda_0| - 1} \stackrel{\eqref{aux-exam-5}}{<} \varepsilon.
		\end{split}
	\end{align*}
	From this, and using that $$E_{j_p} =\{n_1^{j_p} < n_2^{j_p} < n_3^{j_p} < \cdots\} \in \mathscr{F},$$
it follows that $T$ is ${\mathscr{F}}_{\mathbb{C}}$-hypercyclic.%the proof of the necessity part is complete.
\end{proof}

Achieving our first purpose, next we prove that the inclusions in \eqref{inclusions} can be strict.

\begin{corollary} \label{corollary-1}
	\begin{enumerate}[label=(\Alph*)]
		\item \label{item-A} There exists a  $\mathscr{U}$-frequently supercyclic operator $T \in \mathscr{L}(c_0(\mathbb N) \oplus \mathbb C)$ that is not frequently supercyclic.
		\item \label{item-B} There exists  a reiteratively supercyclic operator $T \in \mathscr{L}(c_0(\mathbb N) \oplus \mathbb C))$ that is not $\mathscr{U}$-frequently supercyclic.
		\item \label{item-C} For every $1 \leq p < \infty$, there exists a supercyclic operator $T \in \mathscr{L}(\ell_p(\mathbb N) \oplus \mathbb C)$, that is not reiteratively supercyclic.
	\end{enumerate}
\end{corollary}
\begin{proof}
	\ref{item-A} From \cite[Theorem 5]{BayRuz2015}, there exists a weighted backward shift operator $B_{\bf w}$ on $c_0(\mathbb N)$ that is $\mathscr{U}$-frequently hypercyclic but not frequently hypercyclic. Taking $X := c_0(\mathbb N) \oplus \mathbb C$ and $T := B_{\bf w} \oplus Id_{\mathbb C}$, it follows from Theorem \ref{sup-not-freqsup} that $T \in \mathscr{L}(X)$ is $\mathscr{U}$-frequently supercyclic but not frequently supercyclic.

	\ref{item-B} From \cite[Theorem 7]{BesMenetPerisPuig16}, there exists a weighted backward shift operator $B_{\bf w}$ on $c_0(\mathbb N)$ that is reiteratively hypercyclic but not $\mathscr{U}$-frequently hypercyclic. The result follows as in the proof of \ref{item-A}.

	\ref{item-C} In a similar fashion of the proofs of \ref{item-A} and \ref{item-B}, it is essential to identify a hypercyclic operator on $\ell_p(\mathbb N)$ that is not reiteratively hypercyclic. Once this is done, it is enough to apply Theorem \ref{sup-not-freqsup}. The existence of such an operator is proved in \cite[Theorem 13]{BesMenetPerisPuig16}. More precisely, it has been proven that the weighted shift operator $B_{\bf w}$ on $\ell_p(\mathbb N)$ with $w_k = ((k+1)/k)^{1/p}$ is a mixing operator (hence hypercyclic) that is not reiteratively hypercyclic. \end{proof}

It is usual in linear dynamics to ask
      whether the inverse of an invertible map satisfying a certain dynamical property also satisfies the same property. Such inquiries have been studied in recent years for frequently and $\mathscr{U}$-frequently hypercyclic operators (cf. \cite{Menet20, Menet21}). Next, we provide a negative solution concerning the inverses of frequently and $\mathscr{U}$-frequently supercyclic operators.

\begin{corollary}
	\begin{enumerate}[label=(\Alph*)]
		\item \label{item-A-2} For every $1 \leq p < \infty$, there exists an invertible $\mathscr{U}$-frequently supercyclic operator on $\ell_p(\mathbb N) \oplus \mathbb C$  whose inverse is not $\mathscr{U}$-frequently supercyclic.
		\item \label{item-B-2} There exists an invertible frequently supercyclic operator on $\ell_1(\mathbb N) \oplus \mathbb C$ whose inverse is not frequently supercyclic.
	\end{enumerate}
\end{corollary}
\begin{proof}
	\ref{item-A-2} From \cite{Menet20}, there exists an invertible $\mathscr{U}$-frequently hypercyclic operator $S$ on $\ell_p$ whose inverse $S^{-1}$ is not $\mathscr{U}$-frequently hypercyclic. It follows from Theorem \ref{sup-not-freqsup} that $T := S \oplus Id_{\mathbb C}$ is $\mathscr{U}$-frequently supercyclic on $\ell_p \oplus \mathbb C$, and that its inverse, $T^{-1} = S^{-1} \oplus Id_{\mathbb C}$, is not $\mathscr{U}$-frequently supercyclic on $\ell_p \oplus \mathbb C$. %\textcolor{blue}{\sout{This concludes the proof of part \ref{item-A-2}}}.

	\ref{item-B-2} From \cite{Menet21}, there exists an invertible frequently hypercyclic operator $S$ on $\ell_1$ whose inverse $S^{-1}$ is not frequently hypercyclic. The proof follows, \textit{mutatis mutandis}, as in the proof of \ref{item-A-2}.
\end{proof}

\section{The $\mathscr{F}_\Gamma$-hypercyclicity criterion} \label{Sc-3}

To make it easier for the reader, we recall the definition of collections of series that are unconditionally convergent uniformly over an index set $I$.

\begin{definition}[{\cite[Remark 1]{BesMenetPerisPuig16}}]
	Let $X$ be a Banach space, let $I$ be a non-empty set and let $(x_{n,k})_{n,k} \subset X^{\mathbb N_0 \times I}$ be given. The collection of series $\sum\limits_{n=0}^\infty x_{n,k}$, $k \in I$, is said to be \textit{unconditionally convergent uniformly} in $k \in I$ if, for every $\varepsilon > 0$, there exists  $N \geq 1$ such that for every $k \in I$ and every finite set $F \subset [N,\infty[$, $$\biggl\| \sum_{n \in F} x_{n,k} \biggl\|_X < \varepsilon.$$
\end{definition}

The proof of the $\mathscr{F}_\Gamma$-hypercyclicity criterion below is strongly based on the proof of the $\mathscr{F}$-hypercyclicity criterion \cite[Theorem 3]{BesMenetPerisPuig16}. For the necessary modifications in the original proof to be clear, %, to introduce the $\mathscr{F}_\Gamma$-hypercyclicity criterion. We have closely followed \cite[Theorem 3]{BesMenetPerisPuig16}, making some slight necessary modifications. For the sake of completeness,
we have decided to include the whole proof.

For a given subset $\Gamma$ of $\mathbb{C}^*$, we define $\Gamma^{-1} :=\{a^{-1} : a \in \Gamma\}$. The symbol $\mathbbm{1}_A$ denotes the indicator function of $A \subset \mathbb{N}_0$ mapping $\mathbb{N}_0$ to itself, defined by $\mathbbm{1}_A(n) = 1$ if $n \in A$ and $\mathbbm{1}_A(n) = 0$ if $n \notin A.$

\begin{theorem}[{$\mathscr{F}_\Gamma$-Hypercyclicity Criterion}] \label{hypercyclicity-criterion}
	Let $X$ be a Banach space, let $T \in \mathscr{L}(X)$ and ${(\gamma_n)}_{n \geq 0} \subset \mathbb{C}^*$ be given, and let $\mathscr{F}$ be a Furstenberg family. %\textcolor{violet}{\sout{, and let $\mathscr{F}$ be a hypercyclicity set.}}
Suppose that there exist a dense subset $\mathcal{D}$ of $X$, maps $S_n \colon  \mathcal{D} \to X$, $n \geq 0$, and pairwise disjoint sets $A_k \in \mathscr{F}$, $k \geq 1$, satisfying the following conditions for each $y \in \mathcal{D}$:
	\begin{enumerate}[label=(\Roman*)]
		\item \label{item-1} There exists $k_0 \geq 1$ such that $\sum\limits_{n =0}^{\infty} \mathbbm{1}_{A_k}(n) \, \gamma_n \, S_ny$ converges unconditionally in $X$, uniformly in $k \geq k_0$.
		
		\item \label{item-2} For all $k_0 \in \mathbb N$ and $\varepsilon > 0$, there exists $k \geq k_0$ such that
		
		$\bullet$  $\biggl\|\sum\limits_{n \in F} \dfrac{\gamma_n}{\gamma_q} T^q S_n y\biggl\| \leq \varepsilon$ if $F \subset A_k$ is finite and $q \in \bigcup\limits_{\ell \geq 1} A_\ell \setminus F$,
		
		and further, for any $\delta > 0$, there exists $\ell_0 \geq 1$ such that
		
		$\bullet$ $\biggl\|\sum\limits_{n \in F} \dfrac{\gamma_n}{\gamma_q} T^q S_n y\biggl\| \leq \delta$ if $F \subset A_k$ is finite and $q \in \bigcup\limits_{\ell \geq \ell_0} A_\ell \setminus F$.
		
		\item \label{item-3} $\sup\limits_{q \in A_k} \|T^q S_q y - y\| \to 0$ as $k \to +\infty$.
	\end{enumerate}
	Then, $T$ is $\mathscr{F}_{{\Gamma}^{-1}}$-hypercyclic, where $\Gamma = \{\gamma_{n} : n \in \mathbb N_0\}$.
\end{theorem}
\begin{proof}
	Since $X$ is separable, there exists a set $\{y_\ell : \ell \in \mathbb N\} \subset \mathcal{D}$ which is dense in $X$. Up to a subsequence, we can assume without loss of generality that $A_k \subset [k,\infty[$ for each $k \in \mathbb N$. Combining assumptions \ref{item-1}, \ref{item-2} and \ref{item-3}, we can recursively establish the existence of an increasing sequence ${(k_\ell)}_{\ell \geq 1}$ of positive integers such that for any $\ell \geq 1$, the following holds:
	\begin{align}
		&\biggl\|\sum_{n \in F \cap A_{k_j}} \gamma_n \, S_n y_j\biggl\| \leq \dfrac{1}{\ell \, 2^\ell} \ \ \mbox{if $F \subset [k_\ell,\infty[$ is finite and $j \leq \ell$}, \label{aux-equation-15}\\
		&\biggl\| \sum_{n \in F \setminus \{q\}} \dfrac{\gamma_n}{\gamma_q} \, T^q S_n y_\ell\biggl\| \leq \dfrac{1}{2^\ell} \ \ \mbox{if $F \subset A_{k_\ell}$ is finite and $q \in {\textstyle\bigcup\limits_{i \geq 1}} A_i$}, \label{aux-equation-16}\\
		&\biggl\| \sum_{n \in F} \dfrac{\gamma_n}{\gamma_q} \, T^q S_n y_j\biggl\| \leq \dfrac{1}{\ell \, 2^{\ell}} \ \ \mbox{if $j < \ell$, $F \subset A_{k_j}$ is finite, and $q \in A_{k_\ell}$},\label{aux-equation-17}\\
		&\|T^q S_q y_\ell - y_\ell\| \leq \dfrac{1}{2^\ell} \ \ \mbox{if $q \in A_{k_\ell}$}. \label{aux-equation-18}
	\end{align}
More precisely, condition \ref{item-1} implies \eqref{aux-equation-15}, condition \ref{item-3} confirms the validity of \eqref{aux-equation-18} and condition \ref{item-2} guarantees that  \eqref{aux-equation-16} and \eqref{aux-equation-17} hold.

From \eqref{aux-equation-15} and the fact that $A_k \subset [k, \infty[$, we get
	\begin{eqnarray} \label{aux-equation-19}
		\biggl\|\sum_{n \in F \cap A_{k_j}} \gamma_n \, S_n y_j\biggl\| \leq \dfrac{1}{j \, 2^j} \ \ \mbox{if $F \subset \mathbb N_0$ is finite and $j \geq 1$}.
	\end{eqnarray}
For any $n \in A := \bigcup\limits_{\ell \geq 1} A_{k_\ell}$, let us define $z_n = y_\ell$ if $n \in A_{k_\ell}$. Note that, for any $\ell \geq 1$, and any finite set $F \subset [k_\ell, \infty[$,
	\begin{align} \label{aux-equation-20}
		\begin{split}
			\biggl\|\sum_{n \in F \cap A} \gamma_n \, S_n  z_n\biggl\| \leq \sum_{j=1}^\infty \biggl\| \sum_{n \in F \cap A_{k_j}} \gamma_n \, S_n y_j\biggl\|  \stackrel{\eqref{aux-equation-15}+\eqref{aux-equation-19}}{\leq} \sum_{j=1}^\ell \dfrac{1}{\ell \, 2^\ell} + \sum_{j=\ell+1}^\infty \dfrac{1}{j \, 2^j} \leq \dfrac{1}{2^{\ell-1}}.
		\end{split}
	\end{align}
As $k_\ell \to +\infty$ and $1/2^{\ell-1} \to 0$ as $\ell \to +\infty$, based on \eqref{aux-equation-20}, the series
	$$x = \sum_{n \in A} \gamma_n \, S_n z_n$$
	converges unconditionally. Therefore, the vector $x \in X$ is well-defined.

	To conclude the proof, it suffices to show that $x$ is a $\mathscr{F}_{\Gamma^{-1}}$-hypercyclic vector. For this purpose, consider an arbitrary $\ell \geq 1$. For $q \in A_{k_\ell}$, we have
	\begin{align} \label{aux-equation-21}
		\gamma_q^{-1} \, T^q x - y_\ell = \sum_{\substack{n \in A \\ n < q}} \dfrac{\gamma_n}{\gamma_q} \, T^q S_n z_n + \sum_{\substack{n \in A \\ n > q}} \dfrac{\gamma_n}{\gamma_q} \, T^q S_n z_n + T^q S_q y_\ell - y_\ell.
	\end{align}
Concerning the second term on the right-hand side of the equation above, for $m > q$ we have
	\begin{align*}
		%\small
\begin{split}
			\Biggl\|\sum_{\substack{n \in A \\ q < n \leq m}} \dfrac{\gamma_n}{\gamma_q} \, T^q S_n z_n\Biggl\| \leq \sum_{j=1}^\infty \Biggl\|\sum_{\substack{n \in A_{k_j} \\ q < n \leq m}} \dfrac{\gamma_n}{\gamma_q} \, T^q S_n y_j\Biggl\| \stackrel{\eqref{aux-equation-16}+\eqref{aux-equation-17}}{\leq} \sum_{j=1}^{\ell-1} \dfrac{1}{\ell \, 2^\ell} + \sum_{j=\ell}^\infty \dfrac{1}{2^j} \leq \dfrac{1}{2^{\ell-2}}.
		\end{split}
	\end{align*}
Therefore,	\begin{align} \label{aux-equation-23}
		\begin{split}
			\Biggl\| \sum_{\substack{n \in A \\ n > q}} \dfrac{\gamma_n}{\gamma_q} \, T^q S_n z_n\Biggl\| \leq \dfrac{1}{2^{\ell-2}}.
		\end{split}
	\end{align}
Likewise, the first term on the right-hand side of \eqref{aux-equation-21} satisfies
	\begin{align} \label{aux-equation-22}
		\small \begin{split}
			\Biggl\|\sum_{\substack{n \in A \\ n < q}} \dfrac{\gamma_n}{\gamma_q} \, T^q S_n z_n\Biggl\| < \dfrac{1}{2^{\ell-2}}.
		\end{split}
	\end{align}
Combining \eqref{aux-equation-18}, \eqref{aux-equation-23}, and \eqref{aux-equation-22}, we deduce that
	\begin{align*}
		\begin{split}
			\|\gamma_q^{-1} \, T^q x - y_\ell\| \leq \dfrac{2}{2^{\ell-2}} + \dfrac{1}{2^{\ell}}.
		\end{split}
	\end{align*}
Since $\{y_{\ell} : \ell \in \mathbb N\}$ is dense in $X$ and $A_{k_\ell} \in \mathscr{F}$, the proof is concluded.
\end{proof}

Inspired by \cite[Definition 1]{BesMenetPerisPuig16} we introduce the notion of hypercyclicity set. %\textcolor{blue}{\sout{in $\mathscr{P}(\mathbb N_0)$}.}

\begin{definition} \label{def-hypercyclicity-set}
	A family $\mathscr{F} \subset \mathscr{P}(\mathbb N_0)$ is termed a {\it hypercyclicity set} if it satisfies the following two conditions:
	\begin{enumerate}[label=(\roman*)]
		\item \label{item-1i-def-hyp-set} $\mathscr{F}$ is a Furstenberg family.
		\item \label{item-ii-def-hyp-set} There exists a sequence ${(A_k)}_{k\geq1}$ of pairwise disjoint sets in $\mathscr{F}$ such that for any $j \in A_k$ and $j' \in A_{k'}$ with $j \neq j'$, we have $|j-j'| \geq \max\{k,k'\}$.
	\end{enumerate}
\end{definition}

\begin{remark} \label{rmk-hp-set}
According to \cite[Proposition 1]{BesMenetPerisPuig16},  condition $\ref{item-ii-def-hyp-set}$ of Definition \ref{def-hypercyclicity-set} is a prerequisite for an operator $T$ to be $\mathscr{F}$-hypercyclic in the context of Banach spaces. This indirectly implies that the Furstenberg families $\underline{\mathscr{D}}, \overline{\mathscr{D}}, \overline{\mathscr{B}}$ and $\mathscr{A}_\infty$ are hypercyclicity sets, given it is well-known that there exist Banach spaces supporting frequently hypercyclic operators.
\end{remark}

Now we write the criterion exactly as it will be applied in Sections \ref{Sc-4} and \ref{Sc-5}.

\begin{corollary} \label{thm-criterion}
Let $X$ be a Banach space, let $T \in \mathscr{L}(X)$ and  let ${(\gamma_n)}_{n \geq 0} \subset \mathbb{C}^*$ be given, and let $\mathscr{F}$ be a  hypercyclicity set. If there exist a dense subset $\mathcal{D}$ of $X$ and a map $S \colon \mathcal{D} \to \mathcal{D}$ such that, for each $y \in \mathcal{D}$,
	\begin{enumerate}[label=(\roman*)]
		\item \label{item-0} {{$\sum\limits_{n=0}^{\infty} \left\|\gamma_n \, S^ny\right\| < \infty$,}}
		\item \label{item-i} {$\sup\limits_{r \geq m} \left(\sum\limits_{n=0}^{r-m} \left\|\dfrac{\gamma_n}{\gamma_r} \ T^{r-n}y\right\|\right) + \sup\limits_{r \geq 1}\left(\sum\limits_{n = m + r}^\infty \left\| \dfrac{\gamma_n}{\gamma_r} \, S^{n-r}y\right\|\right) \xrightarrow{m \to \infty} 0$},
		\item \label{item-ii}$T S y = y$,
	\end{enumerate}
	then $T$ is $\mathscr{F}_{\Gamma^{-1}}$-hypercyclic, where $\Gamma = \{\gamma_n : n \in \mathbb N_0\}$.
\end{corollary}
\begin{proof}
	To apply Theorem \ref{hypercyclicity-criterion}, we define $S_n \colon \mathcal{D} \to \mathcal{D}$ by $S_n:=S^n$ for $n \geq 0$. Given that $\mathscr{F} \subset \mathscr{P}(\mathbb N_0)$ is a hypercyclicity set, we can choose a collection of pairwise disjoint sets $\{A_k : k \in \mathbb N\}$  in $\mathscr{F}$ such that for any $j \in A_k$ and any $j' \in A_{k'}$ with $j \not= j'$, we have
	\begin{align} \label{aux-equation-27}
		|j - j'| \geq \max\{k,k'\}.
	\end{align}
Passing to a subsequence if necessary, we can assume, without loss of generality, that $A_k \subset [k, \infty[$ for each $k \in \mathbb N$.

We are in the position to	
	show that conditions \ref{item-1}, \ref{item-2}, and \ref{item-3} of Theorem \ref{hypercyclicity-criterion} hold. To do this, fix $y \in \mathcal{D}$. Condition \ref{item-3} follows easily from  assumption \ref{item-ii}. %\textcolor{blue}{\sout{Next, we will proceed to prove the validity of Conditions \ref{item-1} and \ref{item-2}.}}

	To establish condition \ref{item-1}, let $\varepsilon > 0$ be arbitrary. From \ref{item-0} we can deduce the existence of $k_0 \in \mathbb{N}$ such that, for every $k \geq k_0$,
	\begin{align} \label{aux-equation-25}
		\sum_{n = k_0}^{\infty} \|\gamma_n \, S_n y\| < \varepsilon.
	\end{align}
	Therefore, if $F \subset [k_0, \infty[$ is finite and $k \geq k_0$, then
	\begin{align}
		\begin{split}
			\Biggl\|\sum_{n \in F} \mathbbm{1}_{A_k}(n) \, \gamma_n \, S^n y\Biggl\| \leq \sum_{n \in F} \|\gamma_n \, S^n y \| \stackrel{\eqref{aux-equation-25}}{<} \varepsilon.
		\end{split}
	\end{align}	
This concludes the proof of condition (I).

Finally, to show that \ref{item-2} holds, let $k_0 \in \mathbb N$ and $\varepsilon > 0$ be given. It follows from assumption \ref{item-i} that there exists $m_0 \geq k_0$ such that
	\begin{align} \label{aux-equation-28}
		\begin{split}
			\sup_{r \geq m_0} \left(\sum_{n=0}^{r-m_0} \Biggl\|\dfrac{\gamma_n}{\gamma_r} \, T^{r-n} y\Biggl\|\right) + \sup_{r \geq 1} \left(\sum_{n = m_0+r}^{\infty} \Biggl\| \dfrac{\gamma_n}{\gamma_r} \, S^{n-r} y\Biggl\|\right) < \varepsilon.
		\end{split}
	\end{align}
	We define $k := m_0$. If $F \subset A_k$ is finite and $q \in \bigcup\limits_{\ell \geq 1}A_\ell \setminus F$, then
	\begin{align} \label{aux-equation-29}
		 \begin{split}
			\Biggl\|\sum_{ n \in F} \dfrac{\gamma_n}{\gamma_q} \, T^q \, S_n y\Biggl\| &\leq \sum_{ \substack{n \in F \\ n < q}} \Biggl\|\dfrac{\gamma_n}{\gamma_q} \, T^{q-n} y\Biggl\| + \sum_{\substack{ n \in F \\ n > q}} \Biggl\|\dfrac{\gamma_n}{\gamma_q} \, S^{n-q} y\Biggl\| \\ &\stackrel{(\star)}{\leq} \sum_{0 \leq n \leq q-m_0} \Biggl\|\dfrac{\gamma_n}{\gamma_q} \, T^{q-n} y\Biggl\| + \sum_{n \geq m_0+q} \Biggl\|\dfrac{\gamma_n}{\gamma_q} \, S^{n-q} y\Biggl\| \stackrel{\eqref{aux-equation-28}}{<} \varepsilon,
		\end{split}
	\end{align}
	where $(\star)$ comes from \eqref{aux-equation-27} because $n \in F \subset A_k$ and $q \in A_{\ell'} \setminus F$ for some $\ell' \in \mathbb N$ imply that $|q-n| \geq \max\{k,\ell'\}$. Furthermore, by selecting an arbitrary $\delta > 0$ and using assumption \ref{item-i}, there is $m_1 \geq 1$ such that  \eqref{aux-equation-28} holds with $m_1$ and $\delta$ instead of $m_0$ and $\varepsilon$, respectively. Setting $\ell_0 := m_1$, and reasoning similarly to \eqref{aux-equation-29}, we deduce that
	\begin{align}
		\begin{split} \label{aux-equation-30}
			\displaystyle\biggl\|\sum_{n \in F} \dfrac{\gamma_n}{\gamma_q} T^q S_n y\biggl\| \leq \delta \ \ \mbox{if $F \subset A_k$ is finite and $q \in \bigcup\limits_{\ell \geq \ell_0} A_\ell \setminus F$}.
		\end{split}
	\end{align}
Combining \eqref{aux-equation-29} and \eqref{aux-equation-30} we conclude that condition \ref{item-2} is satisfied, which completes the proof.
\end{proof}

\section{Applications of the $\mathscr{F}_\Gamma$-hypercyclicity criterion} \label{Sc-4}

\subsection{Unilateral pseudo-shift operators} \label{Sub-sec-4.1}

In this subsection, $Z$ is either $c_0(\mathbb N)$ or $\ell_p(\mathbb N)$ for some $1 \leq p < \infty$. Let $\{e_n : n \geq 1\}$ be the canonical basis for $Z$. For a bounded sequence of positive numbers ${\bf w} = {(w_n)}_{n\geq 1}$ and an increasing function $\varphi : \mathbb N \to \mathbb N$ with $\varphi(1) > 1$, the \textit{unilateral pseudo-shift} operator $T_{\varphi,{\bf w}}$ on $Z$ is defined by
\begin{align} \label{pseudo-shift}
	\begin{split}
		T_{\varphi,{\bf w}}\left(\sum_{n=1}^{\infty} x_n \, e_n\right) = \sum_{n=1}^\infty w_{\varphi(n)} \, x_{\varphi(n)} \, e_n.
	\end{split}
\end{align}
Such operators encompass, as particular cases, the standard backward weighted shift operators $B_{\bf v}$, ${\bf v} = (v_k)_{k \geq 1} \subset \mathbb{R}_{> 0}$, as well as their $k$-th powers. For instance, the operator $B_{\bf v}$ can be recovered by setting
	$$\varphi(n) = n + 1 \ \ \mbox{and} \ \ {\bf w} = (1,v_1,v_2, \ldots).$$
 Its $k$-th power $B_{{\bf v}}^k$ can be recovered by setting
 $$\varphi(n) = n+k \ \ \mbox{and} \ \ {\bf w} = \bigg(1,\stackrel{(k)}{\ldots}, 1, \prod_{i=1}^{k} v_i, \prod_{i=2}^{k+1} v_i, \prod_{i=3}^{k+2} v_i, \ldots\bigg).$$

A first application of Corollary \ref{thm-criterion} shall establish that every unilateral pseudo-shift operator on $Z$ is an $\mathscr{F}_\Gamma$-hypercyclic operator whenever $\mathscr{F}$ is a hypercyclicity set and $\Gamma$ is an unbouded set in $\mathbb C$. To prove this result we need the following lemma.%This result is presented as a theorem, and we will require the following  lemma to prove it.

\begin{lemma}
	There is a decreasing sequence ${(a_k)}_{k \geq 0}$ in $(0,1)$ such that
	\begin{align} \label{sequence-lemma}
		a_k^{-1} \cdot \sum_{n = k+1}^\infty a_n \xrightarrow{k \to \infty} 0.
	\end{align}
	In particular, the series $\sum\limits_{k=0}^\infty a_k $ converges.
\end{lemma}
\begin{proof}
	It is enough to prove the existence of a decreasing sequence ${(a_k)}_{k \geq 0}$ in $\mathbb R_{>0}$ such that
	\begin{eqnarray} \label{suf-cod-seq-lem-2}
		0 < \dfrac{1-a_0 - a_1 - \cdots - a_k}{a_k} \leq 2^{-k} \mbox{~ for every } k \geq 0.
	\end{eqnarray}
	We will proceed by induction on $k$. Initially, setting $a_0 = 1/2$ we have that \eqref{suf-cod-seq-lem-2} holds for $k=0$. Now, assuming that $a_0 > a_1 > \cdots > a_k$ have been chosen satisfying \eqref{suf-cod-seq-lem-2}, let us select $a_{k+1}$ as a positive number less than and sufficiently close to $s_k := 1 - a_0 - \cdots - a_k > 0$. In this case,
	$$\dfrac{s_k - a_{k+1}}{a_{k+1}} \leq 2^{-(k+1)}.$$
	This selection is feasible because $\delta^{-1}(s_k-\delta) \to 0$ as $\delta \to s_k$. Note that with this choice, we have $a_{k+1} < a_k$ because
\begin{align*}
	a_{k+1} < s_k \stackrel{\eqref{suf-cod-seq-lem-2}}{\leq} 2^{-k} \, a_k \leq a_k.
\end{align*}
The induction process is complete, and the proof is concluded.
\end{proof}

\begin{theorem} \label{pps-Bw}
For any  hypercyclicity set $\mathscr{F}$ and any unbounded subset $\Gamma$ of $\mathbb{C}$, every unilateral pseudo-shift operator on Z is $\mathscr{F}_{\Gamma}$-hypercyclic.
\end{theorem}
\begin{proof}
	Let $T_{\varphi,{\bf w}}$ be as in \eqref{pseudo-shift}. We aim to show that $T_{\varphi, {\bf w}}$ is $\mathscr{F}_\Gamma$-hypercyclic for every fixed unbounded set $\Gamma$ in $\mathbb C$. First, choose a sequence ${(a_k)}_{k \geq 0}$ satisfying \eqref{sequence-lemma}. Then, for each positive integer $i$, define
	\[
v_i  =\left\{
\begin{array}
[c]{c}%
1;\text{ if }w_i > 1\\
w_i;\text{ otherwise }
\end{array}
\right.  .
\]	
Using that $\Gamma$ is an unbounded subset of $\mathbb{C},$ we can select a sequence ${(\alpha_n)}_{n \geq 0}$ in $\Gamma$ and an increasing sequence ${(k_n)}_{n \geq 0}$ in $\mathbb{N}$ such that, for each $n \in \mathbb N_0$,
	\begin{align} \label{rel-betw-gamma-vs}
		\dfrac{1}{a_{k_n} v_1^n v_2^n \cdots v_{\Phi(n)}^{n}} < |\alpha_n| < \dfrac{1}{a_{k_{n+1}} v_1^{n+1} v_2^{n+1} \cdots v_{\Phi(n+1)}^{n+1}},
	\end{align}
	where $\Phi(n) := \varphi^{n}(n)$. To apply Corollary \ref{thm-criterion}, set $\mathcal{D} := \mbox{span}\{e_n : n \in \mathbb N\}$, $\gamma_n := \alpha_n^{-1}$ for every $n \geq 0$, and define $S \colon \mathcal{D} \to \mathcal{D}$ by putting  $S(e_n) = w_{\varphi(n)}^{-1} \cdot e_{\varphi(n)}$ and extending linearly to $\mathcal{D}$. To be consistent with the notation of Corollary \ref{thm-criterion}, we set $T := T_{\varphi,{\bf w}}$. Now we verify conditions $\ref{item-0}$, $\ref{item-i}$, and $\ref{item-ii}$ of Corollary \ref{thm-criterion}. Since condition $\ref{item-ii}$ is clearly satisfied, our attention turns to establishing conditions $\ref{item-0}$ and $\ref{item-i}$. Let us consider
$$y = \sum_{n=1}^N \beta_n \cdot e_n \in \mathcal{D},$$
where $N$ is a positive integer and $\beta_n \in \mathbb C$ for each $n=1,2,\ldots,N$. We remark that, for any $m \geq N$, $T^m y = 0$. Thus, for $r \geq m \geq N$, we have $T^{r-n}y = 0$ if $0 \leq n \leq r-m$. Hence,
	\begin{align}
		\begin{split} \label{cond-ii-eq1}
			\sup_{r \geq m} \left(\sum_{n=0}^{r-m} {\left\|\dfrac{\gamma_n}{\gamma_r} \, T^{r-n}y\right\|} \right)= 0 \ \ \mbox{if $m \geq N$}.
		\end{split}
	\end{align}
On the other hand, for $n > r + 1$,
	\begin{align}
		%\footnotesize	
\begin{split} \label{eq-aux-3}
			 {\left\|\dfrac{\gamma_n}{\gamma_r} \cdot S^{n-r}y\right\|} &=   {\left\|\dfrac{\alpha_n^{-1}}{\alpha_r^{-1}} \cdot S^{n-r}y\right\|} \stackrel{\eqref{rel-betw-gamma-vs}}{\leq}  {\left\|\dfrac{a_{k_n} v_1^{n} \cdots v_{\Phi(n)}^{n}}{a_{k_{r+1}} v_1^{r+1} \cdots v_{\Phi(r+1)}^{r+1}} \cdot S^{n-r}y\right\|} \\&=    {\left\|\dfrac{a_{k_n}}{a_{k_{r+1}}} \left(\prod_{i=1}^{\Phi(r+1)} v_i^{n-r-1}\right)\left(\prod_{j=\Phi(r+1)+1}^{\Phi(n)} v_j^{n}\right) \cdot S^{n-r}y\right\|} \\&=  {\left\|\sum_{\ell=1}^{N} \left[\dfrac{a_{k_n}}{a_{k_{r+1}}} \left(\prod_{i=1}^{\Phi(r+1)} v_i^{n-r-1}\right)\left(\prod_{j=\Phi(r+1)+1}^{\Phi(n)} v_j^{n}\right)\left(\prod_{q=1}^{n-r} w_{\varphi^q(\ell)}^{-1}\right) \right] \beta_{\ell} \, e_{\varphi^{n-r}(\ell)}\right\|} \\ &\leq \sum_{\ell=1}^{N} \dfrac{a_{k_n}}{a_{k_{r+1}}} \left(\prod_{i=1}^{\Phi(r+1)} v_i^{n-r-1}\right)\left(\prod_{j=\Phi(r+1)+1}^{\Phi(n)} v_j^{n}\right) \left(\prod_{q=1}^{n-r} w_{\varphi^q(\ell)}^{-1}\right)|\beta_{\ell}|.
		\end{split}	
	\end{align}
	Note that, by assuming $n \geq N$ and $0 \leq r < n$, we have
	\begin{align} \label{aux-equation-40}
	\Phi(n) = \varphi^{n}(n) \geq \varphi^{n-r}(n) \geq \varphi^{n-r}(N).
	\end{align}
Thus, when $n \geq N$ and $n > r + 1$, it holds
	\begin{align}
		\begin{split} \label{eq-aux-4}
			\left(\prod_{i=1}^{\Phi(r+1)} v_i^{n-r-1}\right)\left(\prod_{j=\Phi(r+1)+1}^{\Phi(n)} v_j^{n}\right) \left(\prod_{q=1}^{n-r} w_{\varphi^q(\ell)}^{-1}\right) \leq 1, \ \ \ell = 1, 2  \ldots, N,
		\end{split}
	\end{align}
because each $w_s^{-1}$ with modulus greater than 1 appears exactly once for $s \in \{\varphi(\ell), \ldots, \varphi^{n-r}(\ell)\}$, and it is canceled out by $v_s = w_s$, which appears at least once as a factor in the product due to \eqref{aux-equation-40}. It follows from \eqref{eq-aux-3} and \eqref{eq-aux-4} that, for any $m \geq N+1$,
	\begin{align}
		\begin{split} \label{des-auxiliar}
			\sum_{n=m+r}^\infty {\left\|\dfrac{\gamma_n}{\gamma_r} \cdot S^{n-r} y\right\|} \leq {\|y\|}_1 \cdot \sum_{n=m+r}^{\infty} \dfrac{a_{k_n}}{a_{k_{r+1}}},
		\end{split}
	\end{align}
where ${\|y\|}_1 := \sum\limits_{\ell=1}^N |\beta_\ell|$. Furthermore, for $m \geq 2$, we have
\begin{align} \label{eq-aux-extra}
	2+r \leq m+r \implies k_{r+1} + 1 \leq k_{r+2} \leq k_{m+r} \implies a_{k_{r+1}} \geq a_{k_{m+r} - 1}.
\end{align}
Therefore, for $m \geq N + 1$,  	
	\begin{align}
		%\small
\begin{split} \label{cond-ii-eq2}
			\sup_{r \geq 1}\left(\sum_{n = m+r}^\infty {\left\|\dfrac{\gamma_n}{\gamma_r} \cdot S^{n-r}y\right\|}\right)  &\stackrel{\eqref{des-auxiliar}}{\leq} {\|y\|}_1 \cdot \sup_{r \geq 1}  \left( a_{k_{r+1}}^{-1} \sum_{n=m+r}^\infty a_{k_n}\right) \\& \stackrel{\eqref{eq-aux-extra}}{\leq} {\|y\|}_1 \cdot \sup_{r \geq 1}\left( a_{k_{m+r}-1}^{-1} \sum_{n=k_{m+r}}^\infty a_n\right) \xrightarrow{m \to \infty} 0,
		\end{split}
	\end{align}
where the convergence comes from the fact that $k_s \geq s$ for each $s \in \mathbb N$ and from \eqref{sequence-lemma}. Combining \eqref{cond-ii-eq1} and \eqref{cond-ii-eq2}, it follows that condition $\ref{item-i}$ of Corollary \ref{thm-criterion} is satisfied.

In order to check condition $\ref{item-0}$,  observe that, for a fixed $m \in \mathbb N$,
	\begin{align}
		%\small
	\begin{split} \label{eq-aux-1}
			\sum_{n=m}^\infty {\|\gamma_n \cdot S^ny\|} &= \sum_{n=m}^\infty \|{\alpha_n^{-1} \cdot S^ny\|} \\ &\stackrel{\eqref{rel-betw-gamma-vs}}{\leq} \sum_{n=m}^\infty {\left\|\sum_{\ell = 1}^N \left[a_{k_n} \left(\prod_{j=1}^{\Phi(n)} v_j^n \right) \left(\prod_{q=1}^n w_{\varphi^q(\ell)}^{-1}\right)\right] \beta_{\ell} \, e_{\varphi^n(\ell)}\right\|} \\ &\leq \sum_{n=m}^\infty \left[\sum_{\ell = 1}^N a_{k_n} \left(\prod_{j=1}^{\Phi(n)} v_j^n \right) \left(\prod_{q=1}^n w_{\varphi^q(\ell)}^{-1}\right) |\beta_\ell| \right].
		\end{split}
	\end{align}
 From \eqref{aux-equation-40} it follows that $\Phi(n) \geq \varphi^n(N)$ if $n \geq N$. Applying an argument similar to the one we used to derive \eqref{eq-aux-4}, we can deduce the following inequality:
	\begin{align} \label{eq-aux-2}
		\left(\prod_{j=1}^{\Phi(n)} v_j^n\right) \left(\prod_{q=1}^n w_{\varphi^q(\ell)}^{-1}\right) \leq 1 \, \mbox{ for all } \, n \geq N \, \mbox{ and } \, \ell \in \{1, 2, \ldots, N\}.
	\end{align}
 Combining \eqref{sequence-lemma}, \eqref{eq-aux-1} and \eqref{eq-aux-2}, for $m \geq N$ it follows that
	\begin{align} \label{cond-i-eq2}
		\sum_{n=m}^\infty {\|\gamma_n \cdot S^ny\|} \leq {\|y\|}_1 \cdot \sum_{n=m}^\infty a_{k_n} \xrightarrow{m \to \infty} 0.
	\end{align}
	This establishes condition $\ref{item-0}$ of  Corollary \ref{thm-criterion} and completes the proof.
\end{proof}

As a direct consequence of Remark \ref{rmk-hp-set} and Theorem \ref{pps-Bw}, we obtain the following corollary.
	
\begin{corollary} \label{cor-Bw}
For any unbounded subset $\Gamma$ of $\mathbb{C}$, every unilateral pseudo-shift operator on $Z$ is frequently $\Gamma$-supercyclic, hence reiteratively and $\mathscr{U}$-frequently $\Gamma$-supercyclic.
\end{corollary}

\subsection{Existence of frequently supercyclic operators} \label{Sub-sec-4.2}

As a second application of Corollary \ref{thm-criterion}, we will now establish that, for any unbounded set $\Gamma \subset \mathbb C$ and any hypercyclicity set $\mathscr{F}$, every separable infinite-dimensional Banach space $X$ supports a $\mathscr{F}_{\Gamma}$-hypercyclic operator. To do so, we need the following folklore lemma:

\begin{lemma}[{\cite[Theorem 1]{Ovsepian-Pelcz}}] \label{lem-bior-sequ}
	For every separable Banach space $X$ there exists a biorthogonal sequence ${(x_n, x_n^*)}_{n \geq 1} \subset X \times X^*$ satisfying the following properties:
	\begin{enumerate}[label=(\roman*)]
		\item $\sup\limits_{n \in \mathbb N}{\|x_n^*\|}_{X^*} < \infty$  and  ${\|x_n\|}_X = 1$ for each $n$.
		\item ${\rm span}\{x_n : n \in \mathbb N\}$ is dense in $X$.
		\item $x = 0$ if $x_n^*(x) = 0$ for each $n$.			
	\end{enumerate}	
\end{lemma}

\begin{theorem} \label{exis-freq-super}
	Let $\Gamma$ be an unbounded set in $\mathbb C$ and let $\mathscr{F}$ be a hypercyclicity set. Then, every separable infinite-dimensional Banach space supports a $\mathscr{F}_{\Gamma}$-hypercyclic operator.
\end{theorem}
\begin{proof}
	Let $X$ be a separable infinite-dimensional Banach space and let ${(x_n, x_n^*)}_{n \geq 1} \subset X \times X^*$ be a biorthogonal sequence as in Lemma \ref{lem-bior-sequ}. Also, let ${\bf u} = {(u_n)}_{n \geq 1}$ be a summable sequence of positive numbers. Under these conditions, it was proved in \cite{AlvCos24} that the operator $T_{\bf u} : X \to X$ given by
	\begin{align} \label{oper-Tw}
		%\begin{split}
			T_{\bf u}(x) := \sum_{n=1}^\infty x_{n+1}^*(x) \, u_n \, x_n
	%	\end{split}
	\end{align}
	is a well-defined supercyclic operator.

	We aim to show that $T_{\bf u}$ is an $\mathscr{F}_\Gamma$-hypercyclic operator for every unbounded set $\Gamma$ in $\mathbb{C}$ and every Furstenberg hypercyclicity set $\mathscr{F}$. This proof follows, \textit{mutatis mutandis}, from the proof of Theorem \ref{pps-Bw} by setting ${\bf w} := (u_1,u_1,u_2,u_3, u_4, \ldots)$, $\varphi(n) := n + 1$, $T_{\varphi, {\bf w}} := T_{\bf u}$, and writting $x_n$ instead of $e_n$.
\end{proof}

As a consequence of Remark \ref{rmk-hp-set} and Theorem \ref{exis-freq-super}, we have the following corollary.

\begin{corollary}
	Every separable infinite-dimensional Banach space supports a frequently (hence, $\mathscr{U}$-frequently and reiteratively) $\Gamma$-supercyclic operator whenever $\Gamma$ is an unbounded subset of $\mathbb C$.
\end{corollary}

\section{A necessary and a sufficient condition for a scalar set to be $\mathscr{F}$-hypercyclic} \label{Sc-5}

For a subset $\Gamma$ in $\mathbb C$, we define
\begin{align*}
	\mathcal{C}_0(\Gamma) &:= \{{(\alpha_k)}_{k \geq 0} \subset \Gamma : |\alpha_k| \searrow   0\}; \ \ \mbox{and} \\
	\mathcal{R}(\Gamma) &:= \Bigl\{{(\alpha_k)}_{k \geq 0} \subset \Gamma \setminus \{0\} : \displaystyle\lim_{k \to \infty} \frac{|\alpha_{k+1}|}{|\alpha_k|} = 0 \Bigl\}.
\end{align*}
A collection $\mathscr{F} \subset \mathscr{P}(\mathbb N_0)$ is said to be \textit{right-invariant} if, for any $E \in \mathscr{P}(\mathbb N_0)$, $E \in \mathscr{F}$ if and only if $p + E \in \mathscr{F}$ for every $p \in \mathbb N_0$.

The main purpose of this section is to prove the following result:

\begin{theorem} \label{thm-carac-freq-scal-set}
	Let $\Gamma$ be a subset of $\mathbb{C}$ and let $\mathscr{F}$ be a hypercyclicity set. Consider the following conditions:
	\begin{enumerate}[label=(\roman*)]
		\item \label{item-thm-i} $\Gamma \setminus \{0\}$ is non-empty, bounded, and bounded away from zero.
		\item \label{item-thm-ii} $\Gamma$ is an $\mathscr{F}$-hypercyclic scalar set.
		\item \label{item-thm-iii} $\Gamma \setminus \{0\}$ is non-empty, bounded, and $\mathcal{C}_0(\Gamma) \subset \mathcal{R}(\Gamma)$.
	\end{enumerate}
Then, the implication $\ref{item-thm-ii} \Rightarrow \ref{item-thm-iii}$ always holds. Furthermore, if $\mathscr{F}$ is right-invariant and $\mathscr{P}(\mathbb{N}_0) \setminus \mathscr{F}$ is closed under finite unions, then \mbox{$\ref{item-thm-i} \Rightarrow \ref{item-thm-ii}$} also holds.
\end{theorem}

\begin{remark}
It is worth mentioning that $\mathscr{P}(\mathbb{N}_0) \setminus \overline{\mathscr{D}}$ and $\mathscr{P}(\mathbb{N}_0) \setminus \overline{\mathscr{B}}$ are closed under finite unions. This is so because $\overline{d}(E \cup F) \leq \overline{d}(E) + \overline{d}(F)$ and $\overline{b}(E \cup F) \leq \overline{b}(E) + \overline{b}(F)$ for all $E, F \in \mathbb{N}_0$. Moreover, it is well-known that $\mathscr{A}_\infty$, $\overline{\mathscr{D}}$, $\underline{\mathscr{D}}$ and $\overline{\mathscr{B}}$ are all right-invariant hypercyclicity sets.
\end{remark}

\begin{remark} \label{rmk-question}
	In \cite[Theorem A]{CharErnMenet16} it was proved that, for $\mathscr{F}=\mathscr{A}_\infty$, it holds $\ref{item-thm-i} \Leftrightarrow \ref{item-thm-ii}$ in Theorem \ref{thm-carac-freq-scal-set}. Additionally, there the authors inquired about the existence of characterization results for $\mathscr{F}$-hypercyclic scalar sets in the context of both frequent and $\mathscr{U}$-frequent hypercyclicity (see \cite[Question 9]{CharErnMenet16}). Notably, Theorem \ref{thm-carac-freq-scal-set} provides partial answers to this question. More precisely, condition \ref{item-thm-iii} in Theorem \ref{thm-carac-freq-scal-set} is a necessary condition for $\Gamma$ to be a $\overline{\mathscr{D}}$-hypercyclic scalar set, {\it a fortiori}, a $\underline{\mathscr{D}}$-hypercyclic scalar set. Conversely, condition \ref{item-thm-i} in Theorem \ref{thm-carac-freq-scal-set} is found to be a sufficient condition for $\Gamma$ to qualify as a $\overline{\mathscr{D}}$-hypercyclic scalar set.
\end{remark}

To prove the sufficient condition for $\Gamma$ to be a $\mathscr{F}$-hypercyclic scalar set, namely implication $\ref{item-thm-i} \Rightarrow \ref{item-thm-ii}$ of Theorem \ref{thm-carac-freq-scal-set}, we need a few lemmas.
%Next, we will prove $\ref{item-thm-i} \Rightarrow \ref{item-thm-ii}$ of Theorem \ref{thm-carac-freq-scal-set}, thereby establishing a sufficient condition for $\Gamma$ to be a $\mathscr{F}$-frequently hypercyclic scalar set. To do this, we need to introduce two lemmas.

\begin{lemma} \label{lemma-hypset-1}
	Let $T \in \mathscr{L}(X)$ be given, let $\mathscr{F}$ be a  hypercyclicity set, and let $\Gamma_1 \supset \Gamma_2 \supset \cdots \supset \Gamma_k \supset \cdots$ be a sequence of non-empty bounded subsets of $\mathbb C$ such that $\Gamma_1$ is bounded away from zero and $\rm{diam}(\Gamma_k) \to 0$ when $k \to \infty$. If, for every $k$, $x_0$ is $\mathscr{F}_{\Gamma_k}$-hypercyclic for $T$, then $x_0$ is $\mathscr{F}$-hypercyclic for $T$.
\end{lemma}
\begin{proof}
	Since $\Gamma_1$ is bounded away from zero, Cantor's Intersection Theorem  guarantees the existence of a (unique) non-zero scalar $\gamma \in \bigcap\limits_{k \geq 1} \overline{\Gamma_k}$. To prove that $x_0$ is $\mathscr{F}$-hypercyclic for $T$, it suffices to show that $T$ is $\mathscr{F}_{\{\gamma\}}$-hypercyclic for $T$.	Let $x \in X$ and $\varepsilon > 0$ be chosen arbitrarily. If
	$$m \in \mathcal{Q}(x,\varepsilon) := {\textstyle{\bigcup\limits_{k\geq1}}} \mathcal{N}_T(\Gamma_k x_0, B(x, \varepsilon/2)),$$ then there exist $k_m \in \mathbb N$ and $\gamma_m \in \Gamma_{k_m} \subset \Gamma_1$ such that
	\begin{align*}
		\begin{split}
			\|\gamma_m \cdot T^m x_0 - x\| < \dfrac{\varepsilon}{2} \implies  \|T^m x_0\| \leq \dfrac{\varepsilon/2+\|x\|}{\displaystyle\inf_{\xi \in \Gamma_1} |\xi|} < \infty,
		\end{split}
	\end{align*}
	where the finiteness results from $\Gamma_1$ being bounded away from zero. Thus,
	\begin{align}  \label{aux-equation-6}
		\sup_{m \in \mathcal{Q}(x,\varepsilon)}\|T^m x_0\| =: C_{x,\varepsilon} < \infty.
	\end{align}
Take $k_0 \in \mathbb N$ satisfying
	\begin{align}
		\begin{split} \label{aux-equation-7}
			{\rm{diam}}\left(\overline{\Gamma_{k_0}}\right) < \dfrac{\varepsilon}{2 \, C_{x,\varepsilon}}.
		\end{split}
	\end{align}
	For $n \in \mathcal{N}_T(\Gamma_{k_0} x_0, B(x,\varepsilon/2))$, there exists $\beta_n \in \Gamma_{k_0} \subset \Gamma_1$ such that
	$$\|\beta_n \cdot T^n x_0 - x\| < \dfrac{\varepsilon}{2},$$
therefore,
	\begin{align*}
		\begin{split}
			\|\gamma \cdot T^n x_0 - x\| &\leq \|\gamma \cdot T^n x_0 - \beta_n \cdot T^n x_0\| + \|\beta_n \cdot T^n x_0 - x\| \\ &\stackrel{\eqref{aux-equation-6}+\eqref{aux-equation-7}}{<} \dfrac{\varepsilon}{2 \, C_{x,\varepsilon}} \cdot C_{x,\varepsilon} + \dfrac{\varepsilon}{2} = \varepsilon.
		\end{split}
	\end{align*}
	Hence,
	\begin{align}\label{aux-equation-8}
		\mathcal{N}_T(\Gamma_{k_0} x_0, B(x,\varepsilon/2)) \subset \mathcal{N}_T(\gamma \, x_0, B(x,\varepsilon)).
	\end{align}
Since $x_0$ is $\mathscr{F}_{\Gamma_{k_0}}$-hypercyclic, we have $\mathcal{N}_T(\Gamma_{k_0}  x_0, B(x, \varepsilon/2)) \in \mathscr{F}$. It follows from \eqref{aux-equation-8} that $\mathcal{N}_T(\gamma \, x_0, B(x,\varepsilon)) \in \mathscr{F}$ because $\mathscr{F}$ is hereditarily upward. This establishes that $x_0$ is $\mathscr{F}_{\{\gamma\}}$-hypercyclic for $T$, as desired.
\end{proof}

The proof of the next subsublemma is  similar to the proof of \cite[Theorem 6.7]{BayMath-book}. We give the proof for the sake of completeness.
\begin{subsublemma} \label{somewhere-implies-Gamma-super}
	Let $\mathscr{F}$ be a right-invariant hypercyclicity set and let $T \in \mathscr{L}(X)$ be a hypercyclic operator. Given $\Gamma \subset \mathbb{C}^*$ and $x_0 \in X$, if there exists a non-empty open set $O \subset X$ such that $\mathcal{N}_T(\Gamma x_0, W) \in \mathscr{F}$ for every non-empty open set $W \subset O$, then $x_0$ is a $\mathscr{F}_\Gamma$-hypercyclic vector for $T$.
\end{subsublemma}
\begin{proof}
	Let $V$ be a non-empty open subset of $X$. Since $T$ is hypercyclic, it is also topologically transitive, so there exists a non-empty open subset $W \subset O$ such that $T^p(W) \subset V$ for some $p \in \mathbb{N}$. Consequently, $p + \mathcal{N}_T(\Gamma x_0, W) \subset \mathcal{N}_T(\Gamma  x_0, V)$, leading to $\mathcal{N}_T(\Gamma  x_0, V) \in \mathscr{F}$ because $\mathscr{F}$ is right-invariant and hereditarily upward. The arbitrariness of the non-empty open set $V$ %Given that $V$ was arbitrarily chosen as a non-empty open subset of $X$, we conclude
gives that $x_0$ is a $\mathscr{F}_{\Gamma}$-hypercyclic vector for $T$, as required.
\end{proof}

\begin{sublemma} \label{sublemma-hypset-2}
	Let $\mathscr{F}$ be a right-invariant hypercyclicity set, let $T \in \mathscr{L}(X)$ be a hypercyclic operator, and let $x_0 \in X$ and $p \in \mathbb N$ be given. For finitely many non-empty subsets $\Gamma_1, \ldots, \Gamma_p$ of $\mathbb C^*$, define $\mathscr{D} := \bigcup\limits_{j=1}^p \mathscr{D}_j$, where
	$$\mathscr{D}_j := \left\{x \in X : \mathcal{N}_T(\Gamma_j  x_0, B(x,2^{-m})) \in \mathscr{F} \, \text{ for any } \, m \in \mathbb{N}\right\}.$$
	If $\mathscr{D}$ is dense in $X$, then there exists $j_0 \in \{1,2,\ldots,p\}$ such that $x_0$ is a $\mathscr{F}_{\Gamma_{j_0}}$-hypercyclic vector for $T$.
\end{sublemma}
\begin{proof} We shall proceed by induction on $p$. If $p=1$, then for each $x \in \mathscr{D} = \mathscr{D}_1$ and each $m \in \mathbb{N}$,
	we have $\mathcal{N}_T(\Gamma_1  x_0, B(x,2^{-m})) \in \mathscr{F}$. Since $\mathscr{D} = \mathscr{D}_1$ is dense in $X$, it follows that $x_0$ is $\mathscr{F}_{\Gamma_1}$-hypercyclic for $T$. Assuming that the result holds for some $p \geq 1$, let $\Gamma_1$, $\ldots$, $\Gamma_{p+1}$ and $\mathscr{D}$, $\mathscr{D}_1,$ $\ldots,$ $\mathscr{D}_{p+1}$ be according to the sublemma's assumptions. If $\mathscr{D}_1 \cup \cdots \cup \mathscr{D}_{p}$ is dense in $X$, then we can infer from the induction hypothesis that there exists $j_0 \in \{1,2,\ldots,p\}$ such that $x_0$ is $\mathscr{F}_{\Gamma_{j_0}}$-hypercyclic for $T$. Let us consider now the case that $\mathscr{D}_1 \cup \cdots \cup \mathscr{D}_{p}$ is not dense in $X$. In this case there exist $x \in X$ and $\eta > 0$ such that
$$B(x, \eta)\, {\textstyle\bigcap} \Biggl({\textstyle\bigcup\limits_{j=1}^{p}} \mathscr{D}_j\Biggl) = \emptyset.$$
However, since $\mathscr{D} = \mathscr{D}_1 \cup \cdots \cup \mathscr{D}_{p+1}$ is dense in $X$, we must have $B(x,\eta) \cap \mathscr{D}_{p+1}$ dense in $B(x,\eta)$.  Thus, for any non-empty open set $W \subset B(x,\eta) $, there exist $y \in \mathscr{D}_{p+1} \cap W$ and a sufficiently large $m \in \mathbb{N}$ such that $B(y, 2^{-m}) \subset W$. This implies that $\mathcal{N}_T(\Gamma_{p+1} x_0, B(y,2^{-m}))$ belongs to $\mathscr{F}$ and is a subset of $\mathcal{N}_T(\Gamma_{p+1} x_0, W)$. Since $\mathscr{F}$ is hereditarily upward, we conclude that $\mathcal{N}_T(\Gamma_{p+1} x_0, W) \in \mathscr{F}$. As $\mathscr{F}$ is a right-invariant hypercyclicity set and $T$ is a hypercyclic operator, by Subsublemma \ref{somewhere-implies-Gamma-super} (with $O= B(x,\eta)$) it follows that $x_0$ is a $\mathscr{F}_{\Gamma_{p+1}}$-hypercyclic vector for $T$. This completes the proof.
\end{proof}

The next simple lemma shall be useful twice.

\begin{lemma}\label{newlemma} Let $\mathscr{F}$ be a Furstenberg family and let $\Gamma \subset \mathbb C$ be given. If $x_0$ is a $\mathscr{F}_\Gamma$-hypercyclic vetor for $T \in \mathscr{L}(X)$, then $x_0$ is $\Gamma$-supercyclic for $T$. %it holds:
%\begin{align} \label{F-hypimpliesG-sup}
%	\mbox{$x_0$ is $\mathscr{F}_\Gamma$-hypercyclic for $T$  $\Longrightarrow$ $x_0$ is $\Gamma$-supercyclic for $T$.}
%\end{align}
\end{lemma}

\begin{proof} Assume, for the sake of contradiction, that $x_0$ is not a $\Gamma$-supercyclic vector for $T$. In this case, there exists a non-empty open set $U$ such that $\mathcal{N}_T(\Gamma x_0, U)$ is finite. Define $V := U \setminus \bigcup\limits_{m \in \mathcal{N}_T(\Gamma x_0, U)} T^m(\mathbb{C} x_0)$, which remains a non-empty open set, because only finitely many straight lines have been removed from $U$. It is clear that $\mathcal{N}_T(\Gamma x_0, V) = \emptyset$. However, since $x_0$ is an $\mathscr{F}_\Gamma$-hypercyclic vector for $T$, it follows that $\mathcal{N}_T(\Gamma x_0, V) \in \mathscr{F}$, leading to a contradiction because $\emptyset \not\in \mathscr{F}$.
\end{proof}

\begin{lemma} \label{lemma-hypset-2}
	Let $\mathscr{F}$ be a right-invariant hypercyclicity set, let $\Gamma \subset \mathbb C$ be such that $\Gamma \setminus \{0\}$ is non-empty, bounded and bounded away from zero, and let $\delta > 0$ be given. If $\mathscr{P}(\mathbb N_0) \setminus \mathscr{F}$ is closed under finite unions and $x_0$ is $\mathscr{F}_{\Gamma}$-hypercyclic for $T \in \mathscr{L}(X)$, then there exists $\Lambda \subset \Gamma$ such that $x_0$ is $\mathscr{F}_{\Lambda}$-hypercyclic for $T$ and ${\rm{diam}}(\Lambda) < \delta$.
\end{lemma}
\begin{proof}	 Without loss of generality, we can assume that $0 \notin \Gamma$. Since $\Gamma$ is a non-empty bounded set in $\mathbb{C}$, it is also a non-empty totally bounded set. Therefore, there exist elements $\gamma_1,\ldots,\gamma_p$ in $\Gamma$ such that
	\begin{align} \label{aux-equation-12}
		\begin{split}
			\Gamma = {\textstyle \bigcup\limits_{j=1}^p} \Gamma_j \ \ \mbox{where} \ \ \Gamma_j := \{\alpha \in \Gamma : |\alpha - \gamma_j| < \delta\}.
		\end{split}
	\end{align}	
For each $j \in \{1,2,\ldots,p\}$, we define
$$\mathscr{D}_j := \{x \in X : \mathcal{N}_T(\Gamma_j  x_0, B(x,2^{-m})) \in \mathscr{F} \, \mbox{ for any } \, m \in \mathbb N\}.$$
	Our initial purpose is to prove that
	\begin{align} \label{union-of-Dj}
		\begin{split}
			\mathscr{D}_1 \cup \cdots \cup \mathscr{D}_p = X.
		\end{split}
	\end{align} To do so, let $x \in X$ be arbitrarily chosen. It follows from \eqref{aux-equation-12} that, for each $m \in \mathbb N$,
	\begin{align}
		\begin{split} \label{aux-equation-9}
			\mathcal{N}_T(\Gamma  x_0, B(x,2^{-m})) ={\textstyle \bigcup\limits_{j=1}^p} \mathcal{N}_T(\Gamma_j  x_0, B(x,2^{-m})).
		\end{split}
	\end{align}
Since $x_0$ is $\mathscr{F}_{\Gamma}$-hypercyclic for $T$, for each $m \in \mathbb N$,
	\begin{align}
		\begin{split} \label{aux-equation-10}
			\mathcal{N}_T(\Gamma  x_0, B(x,2^{-m})) \in \mathscr{F}.
		\end{split}
	\end{align}
And since $\mathscr{P}(\mathbb N_0) \setminus \mathscr{F}$ is closed under finite unions, it follows from \eqref{aux-equation-9} and \eqref{aux-equation-10} that, for each $m \in \mathbb N$, there exists $j_m \in \{1,2,\ldots, p\}$ such that
	\begin{align}
		\begin{split} \label{aux-equation-11}
			\mathcal{N}_T(\Gamma_{j_m}  x_0, B(x,2^{-m})) \in \mathscr{F}.
		\end{split}
	\end{align}
	As $\{j_m : m \in \mathbb{N}\} \subset \{1,2, \ldots, p\}$, we can construct an increasing sequence of natural numbers $(m_k)_{k \geq 1}$ such that $j_{m_1} = j_{m_k}$ for every $k \in \mathbb{N}$. Combining \eqref{aux-equation-11} with this observation we get
	$$\mathcal{N}_T(\Gamma_{j_{m_1}}  x_0, B(x,2^{-m_k})) \in \mathscr{F}$$
	for every $k \in \mathbb{N}$. Hence, $x \in \mathscr{D}_{j_{m_1}}$ because $\mathscr{F}$ is hereditarily upward. This establishes \eqref{union-of-Dj}.

	Since $x_0$ is an $\mathscr{F}_{\Gamma}$-hypercyclic vector for $T$, %{\color{violet} \sout{and $\mathscr{F} \subset \mathscr{A}_\infty$},
it follows from %\textcolor{violet}{\eqref{F-hypimpliesG-sup}}
Lemma \ref{newlemma} that $x_0$ is a $\Gamma$-supercyclic vector for $T$. We observe that, for $\Gamma \setminus \{0\}$ being non-empty, bounded, and bounded away from zero, \cite[Theorem 3.1]{CharErnMenet16} asserts that a vector in $X$ is hypercyclic for an operator in $\mathscr{L}(X)$ if it is $\Gamma$-hypercyclic for that operator. %Thus, %by applying \cite[Theorem 3.1]{CharErnMenet16}, we can
An application of  this result gives that $x_0$ is hypercyclic for $T$. Consequently, based on \eqref{union-of-Dj} and Sublemma \ref{sublemma-hypset-2}, there exists $j_0 \in \{1,2,\ldots,p\}$ such that $x_0$ is an $\mathscr{F}_{\Gamma_{j_0}}$-hypercyclic vector for $T$. Defining $\Lambda := \Gamma_{j_{0}}$, the fact that ${\rm diam}(\Lambda) < \delta$ completes the proof.
\end{proof}

The proof of $\ref{item-thm-i} \Rightarrow \ref{item-thm-ii}$ of Theorem \ref{thm-carac-freq-scal-set}  is outlined in the following proposition.

\begin{proposition} \label{suffi-cond-scalar-set}
	Let $\mathscr{F}$ be a right-invariant hypercyclicity set such that $\mathscr{P}(\mathbb N_0) \setminus \mathscr{F}$ is closed under finite unions. If $\Gamma \setminus \{0\}$ is non-empty, bounded and bounded away from zero, then $\Gamma$ is an $\mathscr{F}$-hypercyclic scalar set.
\end{proposition}
\begin{proof}
	Suppose that $X$ is a Banach space and $x_0 \in X$ is an $\mathscr{F}_\Gamma$-hypercyclic vector for $T \in \mathscr{L}(X)$. Our goal is to show that  $x_0$ is an $\mathscr{F}$-hypercyclic vector for $T$. From Lemma \ref{lemma-hypset-2} there exists a decreasing sequence $$\Gamma \setminus \{0\} =: \Gamma_1 \supset \Gamma_2 \supset \cdots \supset \Gamma_k \supset \cdots,$$ with the property that each $\Gamma_k$ is non-empty and ${\rm{diam}}(\Gamma_k) \to 0$ as $k \to +\infty$. Furthermore, for each $k \in \mathbb{N}$, $x_0$ is an $\mathscr{F}_{\Gamma_k}$-hypercyclic vector for $T$.  An application of Lemma \ref{lemma-hypset-1} concludes the proof.
\end{proof}

Before establishing the implication $\ref{item-thm-ii} \Rightarrow \ref{item-thm-iii}$ in Theorem \ref{thm-carac-freq-scal-set}, let us explore one consequence of Proposition \ref{suffi-cond-scalar-set}.

\begin{corollary} \label{corol-Gammasuper&hiper}
	Let $\Gamma \subset \mathbb{C}$ be such that $\Gamma \setminus \{0\}$ is non-empty, bounded, and bounded away from zero, and let $B_{\bf w}$ be either a unilateral (bilateral, respectively) backward weighted shift operator on $\ell_p(\mathbb{N})$ (on $\ell_p(\mathbb{Z})$, $1 \leq p < \infty$, respectively). The following statements are equivalent:
	\begin{enumerate}[label=(\Alph*)]
		\item \label{item-A-3} $B_{\bf w}$ is frequently $\Gamma$-supercyclic.
		\item \label{item-B-3} $B_{\bf w}$ is $\mathscr U$-frequently $\Gamma$-supercyclic.
		\item \label{item-C-3} $B_{\bf w}$ is reiteratively $\Gamma$-supercyclic.
		\item \label{item-D-3} The series $\sum\limits_{n \geq 1} (1/w_1 \cdots w_n)^p$ and $\sum\limits_{n < 0}(w_{-1} \cdots w_n)^p$ are convergent if $B_{\bf w}$ is bilateral. In the case that $B_{\bf w}$ is a unilateral backward weighted shift operator, only the convergence of the first series is required.
	\end{enumerate}
\end{corollary}
\begin{proof} The implications $\ref{item-A-3} \Rightarrow \ref{item-B-3} \Rightarrow \ref{item-C-3}$ follow directly from the definitions.

	To establish $\ref{item-C-3} \Rightarrow \ref{item-D-3}$, let us assume that $B_{\bf w}$ is reiteratively $\Gamma$-supercyclic. Calling on Theorem $\ref{thm-carac-freq-scal-set}$, we get that $B_{\bf w}$ is reiteratively hypercyclic. It follows from \cite[Theorem 12]{BesMenetPerisPuig16} that $B_{\bf w}$ is frequently hypercyclic. Now, the statements about the convergent series follow from \cite[Theorems 3 and 4]{BayRuz2015}.

	Lastly, in order to prove $\ref{item-D-3} \Rightarrow \ref{item-A-3}$, we call on \cite[Theorems 3 and 4]{BayRuz2015} once again to ensure that the conditions on the weights in $\ref{item-C-3}$ imply that $B_{\bf w}$ is frequently hypercyclic. Considering that $\Gamma \setminus \{0\} \neq \emptyset$, it follows that $B_{\bf w}$ must be frequently $\Gamma$-supercyclic, thereby concluding the proof. %\textcolor{red}{(Aqui também está usando vários resultados de outros artigos. Não sei se vale a pena falar essas coisas antes, talvez colocar um resumo destas coisas sobre o $B_{\bf w}$ e dar as referências) } %\textcolor{blue}{Geraldo: prefiro deixar como est\'a.}
\end{proof}

Now we proceed to prove the %The next result, Proposition \ref{necessary-cond-scalar-set}, establishes the
implication $\ref{item-thm-ii} \Rightarrow \ref{item-thm-iii}$ of Theorem \ref{thm-carac-freq-scal-set}. %To prove this proposition, we provide a
The next lemma regarding exponential convergence of sequences shall be required.

\begin{lemma} \label{lemma-exponential-conver}
	Let ${(\alpha_k)}_{k \geq 0}$ be a sequence in $\mathbb{C}^*$ such that
	\begin{eqnarray}	\label{aux-equation-37}
		\limsup\limits_{k \to \infty} \frac{|\alpha_{k+1}|}{|\alpha_k|} > 0 \ \ \mbox{and} \ \ |\alpha_k| \searrow 0 \, \mbox{ as } \, k \to +\infty.
	\end{eqnarray}
	Then, there exist $\delta \in (0,1)$ and a subsequence ${(\alpha_{k_\ell})}_{\ell \geq 0}$ of ${(\alpha_k)}_{k \geq 0}$ such that, for each $\ell \in \mathbb{N}_0$,
	\begin{eqnarray} \label{aux-equation-35}
		\displaystyle	\delta^2 \leq \frac{|\alpha_{k_{\ell+1}}|}{|\alpha_{k_\ell}|} \leq \delta.
	\end{eqnarray}
\end{lemma}
\begin{proof}
	From (\ref{aux-equation-37}) we can take $\delta \in (0,1)$ and an infinite subset $E$ of $\mathbb{N}_0$ such that, for each $k \in E$,
	\begin{eqnarray} \label{aux-equation-36}
		\displaystyle\frac{|\alpha_{k+1}|}{|\alpha_k|} \geq \delta.
	\end{eqnarray}
	Without loss of generality, we consider $E = \mathbb{N}_0$. We choose $k_0 = 0$ and, as $|\alpha_k| \searrow 0$ as $k \to \infty$, we can choose $k_1 > k_0$ as the smallest natural number such that $|\alpha_{k_1}|/|\alpha_{k_0}| < \delta$. Note that $k_1$ must satisfy $k_1 > k_0 + 1$ due to (\ref{aux-equation-36}). With this selection, we obtain
	\begin{eqnarray}
		\displaystyle	\frac{|\alpha_{k_1}|}{|\alpha_{k_0}|} = \frac{|\alpha_{k_1}|}{|\alpha_{k_1-1}|} \cdot \frac{|\alpha_{k_1-1}|}{|\alpha_{k_0}|} \geq \delta^2.
	\end{eqnarray}
Assume that we have chosen $\alpha_{k_0}, \alpha_{k_1}, \ldots, \alpha_{k_{\ell+1}}$ satisfying (\ref{aux-equation-35}). Let $k_{\ell +2} > k_{\ell +1}$ be the smallest natural number such that $|\alpha_{k_{\ell+2}}|/|\alpha_{k_{\ell+1}}| < \delta$. Reasoning as above, it follows that $|\alpha_{k_{\ell+2}}|/|\alpha_{k_{\ell+1}}| \geq \delta^2$. This concludes the proof.
\end{proof}

\begin{proposition} \label{necessary-cond-scalar-set}
	Let $\mathscr{F}$ be a hypercyclicity set. If $\Gamma$ is an $\mathscr{F}$-hypercyclic scalar set, then $\Gamma \setminus \{0\}$ is non-empty, bounded, and $\mathcal{C}_0(\Gamma) \subset \mathcal{R}(\Gamma)$. %{\color{purple} (OBS: Pode existir uma sequência ${(\alpha_k)}_{k \geq 1}$ em $\Gamma \setminus \{0\}$ convergindo para zero?)}\textcolor{blue}{?????????}
\end{proposition}
\begin{proof}
	It is evident that $\Gamma \setminus \{0\}$ cannot be empty.  Suppose that $\Gamma \setminus \{0\}$ is unbounded. From Theorem \ref{pps-Bw} we know that the backward shift operator $B(e_n) = e_{n+1}$ on $\ell_2(\mathbb N)$ is $\mathscr{F}_\Gamma$-hypercyclic, despite not being hypercyclic due to its norm being 1. Thus, $B$ is not $\mathscr{F}$-hypercyclic %{\color{violet} due to \eqref{F-hypimpliesG-sup}}
by Lemma \ref{newlemma}. This contradiction proves that $\Gamma \setminus \{0\}$ is bounded.

	Suppose now that there exists a sequence ${(\alpha_k)}_{k \geq 0}$ in $\mathcal{C}_0(\Gamma) \setminus \mathcal{R}(\Gamma)$. In other words, ${(\alpha_k)}_{k \geq 0}$ satisfies the conditions in  \eqref{aux-equation-37}. By using Lemma \ref{lemma-exponential-conver}, we can find a subsequence of ${(\alpha_k)}_{k \geq 0}$ that satisfies \eqref{aux-equation-35} for some $\delta \in (0,1)$. For the sake of simplicity we still denote this subsequence by ${(\alpha_k)}_{k \geq 0}$. Thus,
	\begin{align} \label{aux-equation-38}
		\delta^2 \leq \dfrac{|\alpha_{k+1}|}{|\alpha_k|} \leq \delta \ \ \mbox{for some $\delta \in (0,1)$ and every $k \in \mathbb{N}_0$.}
	\end{align}
Consider $\eta > 1$ such that $ \delta^2 \cdot \eta > 1$. Define $B_{\bf w}$ as the bilateral weighted shift operator on $\ell_2(\mathbb Z)$ with $w_k = \eta$ for $k > 0$ and $w_k = 1$ otherwise. According to \cite[Example 4.15]{Gro-ErdMang2011-book}, $B_{\bf w}$ is not hypercyclic, hence it is not $\mathscr{F}$-hypercyclic, because $\prod\limits_{i \in A} w_i = 1$ whenever $A \subset \mathbb{Z}_{\leq 0}$ is non-empty. Our task now is to establish that $B_{\bf w}$ is $\mathscr{F}_\Gamma$-hypercyclic, which can be achieved by applying Corollary \ref{thm-criterion}. To do so, define $\gamma_k := \alpha_k^{-1}$ for each $k \in \mathbb N_0$, let $\mathcal{D} := c_{00}(\mathbb Z)$ (which is dense in $\ell_2(\mathbb Z)$), and let $S \colon \mathcal{D} \to \mathcal{D}$ be the (unique) linear mapping such that $S(e_k) := w_k^{-1} \, e_{k+1}$ for every $k$. It is clear that $B_{\bf w} S z = z$ for each $z \in \mathcal D$, showing that condition $\ref{item-ii}$ of Corollary \ref{thm-criterion} is satisfied. To check condition  $\ref{item-i}$, fix
	$$y := \sum_{j=-N}^N y_j \, e_j \in \mathcal{D} \, \mbox{ with } \, N \in \mathbb N.$$
	Note that, for $r \geq m \geq 2$,
	\begin{align}
		\begin{split} \label{aux-equation-13}
			\sum_{n=0}^{r-m}{\left\|\dfrac{\gamma_n}{\gamma_r} \, B_{\bf w}^{r-n} y\right\|}_2 &\leq \sum_{n=0}^{r-m}\dfrac{|\alpha_r|}{|\alpha_n|} \eta^{N} {\|y\|}_2 = \sum_{n=0}^{r-m}\dfrac{|\alpha_{n+1} \cdots \alpha_r|}{|\alpha_n \cdots \alpha_{r-1}|} \eta^{N} {\|y\|}_2 \\ &\stackrel{\eqref{aux-equation-38}}{\leq} \sum_{n=0}^{r-m}  \delta^{r-n} \eta^{N} {\|y\|}_2 \leq \eta^N \, {\|y\|}_2  \sum_{n=m}^\infty \delta^n  \xrightarrow{m \to \infty} 0.
		\end{split}
	\end{align}
	On the other hand, for all $m \geq 1$ and $r \geq 1$,
	\begin{align}
		\begin{split} \label{aux-equation-14}
			\sum_{n=m+r}^{\infty} {\left\|\dfrac{\gamma_n}{\gamma_r} \, S^{n-r} y\right\|}_2 &\leq \sum_{n=m+r}^{\infty} \dfrac{|\alpha_r|}{|\alpha_n|} \, \eta^{r-n+N} {\|y\|}_2  = \sum_{n=m+r}^{\infty} \dfrac{|\alpha_r \cdots \alpha_{n-1}|}{|\alpha_{r+1} \cdots \alpha_n|} \, \eta^{r-n+N} {\|y\|}_2 \\ &\stackrel{\eqref{aux-equation-38}}{\leq} \eta^N \, {\|y\|}_2 \cdot \sum_{n=m+r}^{\infty}  \left(\dfrac{1}{\delta^2 \cdot \eta}\right)^{n-r}  \\ &= \eta^N \, {\|y\|}_2 \cdot \sum_{n=m}^{\infty}  \left(\dfrac{1}{\delta^2 \cdot \eta}\right)^{n} \xrightarrow{m \to \infty} 0.
		\end{split}
	\end{align}
	From \eqref{aux-equation-13} and \eqref{aux-equation-14} we conclude that conditon $\ref{item-i}$ of Corollary \ref{thm-criterion} is satisfied. %	The validity of C
The following calculation shows that condition $\ref{item-0}$ also holds:%in Corollary \ref{thm-criterion} follows from the calculation below:
	\begin{align}
		\begin{split}
			\sum_{n=m}^\infty {\|\gamma_n \cdot S^n y\|}_2 &\leq \sum_{n=m}^{\infty} \dfrac{|\alpha_{n-1} \cdots \alpha_0|}{|\alpha_n \cdots \alpha_1 |} \, \dfrac{1}{|\alpha_0|} \, \eta^{-n+N} {\|y\|}_2 \\ &\stackrel{\eqref{aux-equation-38}}{\leq} \dfrac{\eta^N \, {\|y\|}_2}{|\alpha_0|} \cdot \sum_{n=m}^\infty \left(\dfrac{1}{\delta^2 \cdot \eta}\right)^n \xrightarrow{m \to \infty} 0.
		\end{split}
	\end{align}
It follows that $B_{\bf w}$ is $\mathscr{F}_\Gamma$-hypercyclic. This contradiction completes the proof.
\end{proof}

\section*{Acknowledgement}

The first author extends gratitude to the Instituto de Matem\'atica e Estat\'istica at Universidade Federal de Uberlândia for their gracious hospitality during the author's one-year visit.

\bibliographystyle{amsplain}

\end{document}